\documentclass{amsart}
\pdfoutput=1
\usepackage{amssymb}
\usepackage{amsmath}
\usepackage[alphabetic]{amsrefs}
\usepackage{tikz-cd}
\usepackage{hyperref}
\usepackage{mathtools}
\usepackage{microtype}
\usepackage{lmodern}

\theoremstyle{plain}
\newtheorem{theorem}[subsection]{Theorem}
\newtheorem{proposition}[subsection]{Proposition}
\newtheorem{lemma}[subsection]{Lemma}
\newtheorem{corollary}[subsection]{Corollary}

\theoremstyle{definition}
\newtheorem{definition}[subsection]{Definition}
\newtheorem{example}[subsection]{Example}

\theoremstyle{remark}
\newtheorem{remark}[subsection]{Remark}

\newcommand{\tenscorep}{\mathbin{\begin{tikzpicture}[baseline,x=.75ex,y=.75ex] \draw (-0.8,1.15)--(0.8,1.15);\draw(0,-0.25)--(0,1.15); \draw (0,0.75) circle [radius = 1];\end{tikzpicture}}}

\numberwithin{equation}{section}

\title{Anyonic Quantum Symmetries of Finite Spaces}

\author[Anshu]{Anshu}
\address{Department of Mathematics and Statistics, York University, Ross Bldg, 4700 Keele ST, Toronto, ON M3J 1P3}
\address{Fields Institute for Research in Mathematical Sciences\\
222, College ST, Toronto, ON, M5T 3J1}
\email{anshu@fields.utoronto.ca}
\author[Bhattacharjee]{Suvrajit Bhattacharjee}
\address{Matematisk institutt, Universitetet i Oslo, P.O. Box 1053, Blindern, 0316 Oslo, Norway}
\email{suvrajib@math.uio.no}
\author[Rahaman]{Atibur Rahaman}
\address{Department of Mathematics \& Statistics, Indian Institute of Science Education and Research Kolkata, Mohanpur, 741246, India}
\email{atibur.pdf@iiserkol.ac.in}
\author[Roy]{Sutanu Roy}
\address{School of Mathematical Sciences, National Institute of Science Education and Research Bhubaneswar, Jatni, 752050, India}
\address{Homi Bhabha National Institute, Training School Complex, Anushaktinagar, Mumbai, 400094, India}
\email{sutanu@niser.ac.in}

\begin{document}

\begin{abstract}
We construct a braided analogue of the quantum permutation group and show that
it is the universal braided compact quantum group acting on a finite space in
the category of 
$\mathbb{Z}/N\mathbb{Z}$\nobreakdash-$\textup{C}^*$\nobreakdash-algebras with a
twisted monoidal structure. As an application, we prove the existence of braided
quantum symmetries of finite, simple, undirected, circulant graphs, explicitly
compute it for several examples, and obtain a generalization of a result of
Banica in this direction. Finally, in an appendix, we briefly describe the
irreducible representations of this braided analogue of the quantum permutation group and their fusion rules.       
\end{abstract}



\maketitle

\section{Introduction}

The theory of braided compact quantum groups, as introduced in \cite{mrw2016}, is a new development in the rich and beautiful theory of compact quantum groups~\cite{woropseudo}. One of the motivations behind the development is the desire to allow non-zero complex deformation parameters $q$ in defining $\textup{SU}_q(2)$, constructed by Woronowicz (\cite{worosu}) for $q$ in the unit interval $(0,1)$. Going beyond the real case brings forth new features; the comultiplication $\Delta_{\textup{SU}_q(2)}$ does not take values in the minimal tensor product $\textup{C}(\textup{SU}_q(2)) \otimes \textup{C}(\textup{SU}_q(2))$ anymore. To find out the receptacle of the comultiplication $\Delta_{\textup{SU}_q(2)}$, one has to take into account a hidden $\mathbb{T}$\nobreakdash-action on $\textup{C}(\textup{SU}_q(2))$, plus a twisting by the nontrivial bicharacter on $\mathbb{Z}$ governed by the unit complex number $\zeta=q/\overline{q}$, which roughly speaking, measures how much the case under consideration deviates from the real case. One then has to, accordingly, twist the minimal tensor product $\otimes$ to obtain a new monoidal structure $\boxtimes_{\zeta}$, depending on the parameter $\zeta$.

The formalism, which the twisting $\boxtimes_{\zeta}$ described in the previous paragraph is based on, is the theory of quantum group\nobreakdash-twisted tensor product, or braided tensor product as we call it, of two $\textup{C}^*$\nobreakdash-algebras, laid out by Meyer, the fourth author and Woronowicz in \cites{mrw2014,mrw2016}. Originally introduced by Vaes in his seminal paper \cite{vaes}, this braided tensor product has played a fundamental role in the work of Nest and Voigt \cite{nestvoigt}, making contact with bivariant theory and enabling one to generalize Poincar\'e duality to the equivariant setting - an essential ingredient in the theory of noncommutative manifolds due to Connes \cite{connes}. In \cite{nestvoigt}, the authors start with a (locally compact) quantum group $G$ and two $G$\nobreakdash-Yetter\nobreakdash-Drinfeld $\textup{C}^*$\nobreakdash-algebras $A$ and $B$ and construct $A \boxtimes B$, the braided tensor product of $A$ and $B$. The algebras $A$ and $B$ can equivalently be described as $\textup{D}(G)$\nobreakdash-$\textup{C}^*$\nobreakdash-algebras, where $\textup{D}(G)$ is the Drinfeld double of the quantum group $G$. The Drinfeld double $\textup{D}(G)$ of $G$ carries an $\textup{R}$\nobreakdash-matrix and is the archetypal example of a quasitriangular quantum group. This is the starting point of \cite{mrw2016}; a quasitriangular quantum group $G$ with $\textup{R}$-matrix $\textup{R}$ and two $G$\nobreakdash-$\textup{C}^*$\nobreakdash-algebras $A$ and $B$. The authors then use the braidings associated to the $\textup{R}$\nobreakdash-matrix $\textup{R}$ and constructs the braided tensor product $A \boxtimes_{\textup{R}} B$, generalizing the construction in \cite{nestvoigt}. 

Let then $G$ be a quasitriangular quantum group with a fixed
$\textup{R}$\nobreakdash-matrix $\textup{R}$. The category of
$G$\nobreakdash-$\textup{C}^*$\nobreakdash-algebras (with morphisms as in
Section \ref{sec:bcqg}) becomes a monoidal category with respect to the braided
tensor product $\boxtimes_{\textup{R}}$. A braided compact quantum group over
$G$ then consists of a $G$\nobreakdash-$\textup{C}^*$\nobreakdash-algebra $A$
and a $G$\nobreakdash-equivariant morphism, the comultiplication map, $\Delta
\in \mathrm{Mor}^{G}(A, A \boxtimes_{\textup{R}} A)$ satisfying axioms similar
to those of an ordinary compact quantum group but the minimal tensor product
replaced with $\boxtimes_{\textup{R}}$. We observe that we thus require two
kinds of noncommutativity, as proposed by Majid in \cite{MR1415857}, to define
braided quantum groups. The first kind is called \emph{inner noncommutativity},
which corresponds to the underlying $\textup{C}^*$\nobreakdash-algebra $A$ being
noncommutative, whereas the second kind is the \emph{outer noncommutativity},
corresponding to the noncommutative tensor product $A \boxtimes_{\textup{R}}A$,
that contains the range of the comultiplication map. According to this
dictionary, classical groups are commutative for both kinds. However, the
example of a (locally compact) braided quantum group which is inner commutative
but outer noncommutative was considered by Woronowicz in \cite{MR1647968}. The
quantum $\textup{SU}(2)$ for (nonzero) complex deformation parameter $q$ is the
prime example of a braided compact quantum group over the circle group
$\mathbb{T}$ whose $\textup{R}$-matrix is given by the bicharacter $\textup{R} :
\mathbb{Z} \times \mathbb{Z} \rightarrow \mathbb{T}$, $(m,n) \mapsto
\zeta^{mn}$, recalling that $\zeta=q/\overline{q}$. Several other examples,
again over the circle group $\mathbb{T}$ with the same
$\textup{R}$\nobreakdash-matrix as above, are constructed in \cites{rahamanroy,
  MR2021, R2021, BJR2022}, the latter three being part of a planned series by
the authors. One of the main motivations behind the series is the beautiful
interplay of quantum groups and bivariant theory in \cite{nestvoigt}, which
provides a rich area of thorough exploration. 

Besides the motivations mentioned in the previous paragraph, there is a fruitful
point of view that motivates us further along this line of investigations. It is
the idea of viewing braided compact quantum groups as symmetry objects of
suitable spaces - an idea of deep interest and rich study in the case of
ordinary compact quantum groups initiated by Wang \cite{wang}, see for example
the introduction of \cite{BJR2022}. The space under consideration is a finite
space consisting of $N$ points, together with the action of its symmetry group,
the permutation group $\textup{S}_N$. However, it is well-known that the group
algebra of a nonabelian finite group does not admit any coquasitriangular
structure (see for instance \cite{montgomery}*{Example 10.2.8}), thus forcing us
to restrict our attention to the action of the cyclic group
$\mathbb{Z}/N\mathbb{Z}$ via the cyclic permutation of length $N$. Fixing a primitive $N$\nobreakdash-th root of unity $\omega$ enables us to
construct an $\textup{R}$\nobreakdash-matrix for the group
$\mathbb{Z}/N\mathbb{Z}$. Consequently, the embeddings of
$\mathbb{Z}/N\mathbb{Z}$\nobreakdash-$\textup{C}^*$\nobreakdash-algebras $A$ and
$B$ in $A\boxtimes_{\textup{R}}B$ commute up to $\omega$, the phase
factor. When $N=2$, this braiding closely resembles fermions, giving rise to the
broader scenario when particles with fractional statistics or \emph{anyons} are
permuted, as introduced in \cite{MR1081990}. In his article~\cite{MR1258379},
Majid witnessed a resemblance between anyons and braided Hopf algebras in the
category of $\mathbb{Z}/N\mathbb{Z}$\nobreakdash-algebras. Extending that
dictionary, a braided compact quantum group over $\mathbb{Z}/N\mathbb{Z}$ is an
analytic counterpart of an anyonic quantum group that appeared in Majid's
work. In this article, we call them anyonic compact quantum groups, the
following theorem providing the first such example:

\begin{theorem}
The anyonic quantum permutation group $\textup{S}^+_N(\textup{R})$ exists for any primitive $N$\nobreakdash-th root of unity $\omega$.
\end{theorem}

As one expects, the anyonic quantum permutation group should be a quantum subgroup of an anyonic free unitary quantum group. The difficulty in defining such an analogue, as pointed out in \cite{BJR2022}, is the definition of the conjugate representation of a unitary representation. However, we settled this in \cite{BJR2022} over the circle group $\mathbb{T}$ and a braided free unitary quantum group $\textup{U}^+_{\zeta}(F)$ was constructed. In this article, we first achieve the same construction but over the cyclic group $\mathbb{Z}/N\mathbb{Z}$, i.e., we construct the anyonic free unitary quantum group $\textup{U}^+_{N}(\textup{R})$ and then show that $\textup{S}^+_N(\textup{R})$ is a quantum subgroup, as stated in the theorems below.

\begin{theorem}
The anyonic free unitary quantum group $\textup{U}^+_N(\textup{R})$ exists for any primitive $N$\nobreakdash-th root of unity $\omega$.    
\end{theorem}

\begin{theorem}
The anyonic quantum permutation group $\textup{S}^+_N(\textup{R})$ is a quantum subgroup \textup{(}in the braided sense, see Theorem \textup{\ref{thm:quotient})} of the anyonic free unitary quantum group $\textup{U}^+_N(\textup{R})$. 
\end{theorem}

The theorems stated above join hands in obtaining the main result of this article, which presents the anyonic quantum permutation group $\textup{S}^+_N(\textup{R})$ as the anyonic quantum symmetry group of the finite space $X_N$ consisting of $N$ points.

\begin{theorem}\label{thm:symm}
The anyonic quantum symmetry group of the finite space $X_N$ is the anyonic quantum permutation group $\textup{S}^+_N(\textup{R})$.
\end{theorem}

As it turns out, the $\textup{C}^*$\nobreakdash-algebra
$\textup{C}(\textup{S}^+_3(\textup{R}))$ is commutative but
$\textup{C}(\textup{S}^+_N(\textup{R}))$ is noncommutative for
$N=4$ onwards. Although the $\textup{C}^*$\nobreakdash-algebra
$\textup{C}(\textup{S}^+_3(\textup{R}))$ is commutative, the comultiplication
still takes values in the braided tensor product. It means that
$\textup{S}^+_3(\textup{R})$ is inner commutative but outer noncommutative, and,
from the perspective of quantum symmerties, it lacks an interpretation. 

As an application of the above theorem, we prove the existence of anyonic
quantum symmetries of a class of graphs. More precisely, we obtain the following
theorem.

\begin{theorem}\label{theo:graph}
Let $\Gamma=(E,V)$ be a finite, simple, undirected, circulant graph. Then the anyonic quantum symmetry group $(\mathrm{Qaut}(\Gamma),\eta^{\Gamma})$ of the graph $\Gamma$ exists.
\end{theorem}

We explicitly compute several examples and obtain a generalization of a result
of Banica in this direction, \cite{banicagraphs}*{Theorem 4.2}.

\begin{theorem}\label{thm:dihedral}
Let $\Gamma$ be a circulant graph with $N\geq 5$
vertices. Consider the polynomial $Q(z)=\sum_{i\in\mathbb{Z}_{N}}a_{i}z^{i}$,
where $a_{i}$ is the element in the $i$-th column of the first row of the
adjacency matrix $A_{\Gamma}$. Suppose the numbers 
$$ 
Q(1),Q(\omega),\dots,Q(\omega^{[\frac{N}{2}]}) 
$$are distinct. Then $\textup{C}(\mathrm{Qaut}(\Gamma))\cong\textup{C}(D_{N})$,
where $D_{N}$ is the dihedral group, and the comultiplication map
$\Delta_{\mathrm{Qaut}(\Gamma)}$ takes values in $\textup{C}(\mathrm{Qaut}(\Gamma))\boxtimes_{\textup{R}}\textup{C}(\mathrm{Qaut}(\Gamma))$.
\end{theorem}

Once again, we observe that $\mathrm{Qaut}(\Gamma)$ in the theorem above is 
inner commutative but outer noncommutative. 

The bosonization construction, introduced in \cites{mrw2016,R2023}, based on the
algebraic Radford (\cite{R1985}) and Majid (\cite{MR1257312}) bosonization (which in
turn is based on the semidirect product construction for groups) provides an
equivalence between braided compact quantum groups and a class of ordinary
quantum groups (those with an idempotent quantum group morphism). The
construction of an  ordinary quantum group from a braided quantum group is
essentially a systematic way of turning the underlying ``interacting'' braided
tensor product of the braided quantum group into a ``non-interacting'' ordinary
tensor product, similar to the construction of ``non-interacting'' bosons from
``interacting'' fermions. Due to this behavior, Majid called  this construction
bosonization. To not interrupt the flow of the article, we push the explicit
description of the bosonization(s) of the anyonic quantum permutation group (and
of the anyonic free unitary quantum group), as in \cite{BJR2022} and following
\cite{MR2021}, to an appendix, which contains the following theorem.

\begin{theorem}
The bosonization $\textup{S}_{N}^+(\textup{R}) \rtimes \mathbb{Z}/N\mathbb{Z}$ of $\textup{S}_{N}^+(\textup{R})$ is a compact quantum group such that $\textup{C}(\textup{S}_{N}^+(\textup{R}) \rtimes \mathbb{Z}/N\mathbb{Z})$ is the crossed product of $\textup{C}(\textup{S}_{N}^+(\textup{R}))$ with $\mathbb{Z}/N\mathbb{Z}$, the $\mathbb{Z}/N\mathbb{Z}$-action being given by the automorphism induced by $\omega$.
\end{theorem}

We use this theorem to describe the irreducible representations and the fusion
rule of the bosonization $\textup{S}_{N}^+(\textup{R}) \rtimes
\mathbb{Z}/N\mathbb{Z}$. As is known from \cite{MR2021}, the representation
category of $\textup{S}_{N}^+(\textup{R})$ is equivalent to that of the
bosonization $\textup{S}_{N}^+(\textup{R}) \rtimes \mathbb{Z}/N\mathbb{Z}$, thus
yielding complete knowledge of the representation category.

\begin{theorem}
The anyonic quantum permutation group $\textup{S}_{N}^+(\mathrm{R})$ has irreducible representations $r_{(a,a')}$, $0 \leq a \leq N-1$, $a' \in \mathbb{N}$ such that any irreducible representation is unitarily equivalent to exactly one of these and moreover they satisfy the fusion rule
\begin{equation*}
r_{(a,a')} \tenscorep r_{(b,b')} \cong r_{(a+b,|a'-b'|)} \oplus r_{(a+b,|a'-b'|+1)} \oplus \dots \oplus r_{(a+b,a'+b')}. 
\end{equation*}
\end{theorem}

We now briefly describe the organization of this paper. In Section
\ref{sec:bcqg}, after fixing notations, we begin with recalling the definition
of an anyonic quantum group. We then construct the anyonic quantum permutation
group and the anyonic free unitary quantum group (Definition
\ref{def:permutation} and Definition \ref{def:unitary}, respectively). We then
prove that the anyonic quantum permutation group is a quantum subgroup of the
anyonic free unitary quantum group (Theorem \ref{thm:quotient}). Finally, Section \ref{sec:symm}
defines what we call anyonic symmetry of the finite space and proves Theorem
\ref{thm:symm} above (Theorem \ref{thm:main}). The study of the cases $N=3$ and
$N=4$ are also considered in this section. The anyonic quantum symmetries of
graphs is taken up in Section \ref{sec:graph-symmetry}, in which we prove
Theorem \ref{theo:graph} and Theorem \ref{thm:dihedral} (Theorem
\ref{theo:graph-symmetry} and Theorem \ref{theo:dihedral}, respectively)
mentioned above. The Appendix \ref{appendix:boso} describes the bosonization
and the irreducible representations, along with their fusion rule, of $\textup{S}^+_N(\textup{R})$.

To end this Introduction, let us mention that a recent preprint \cite{cocycle-twist} provides a systematic approach to construct braided analogues of known compact quantum groups, based on the transmutation procedure. The anyonic quantum permutation group constructed in this article as well as the ones constructed in \cites{kmrw2016,BJR2022} may also be obtained via this method, see \cite{cocycle-twist}*{Section 3} for the details.

\subsection*{Acknowledgments} The first author was supported jointly by the
Fields Institute and Prof. Ilijas Farah at the York University. She is also
grateful to the National Institute of Science Education and Research Bhubaneswar 
for offering her a postdoctoral position where major portion of this work took
place. The second author was supported by the NFR project 300837 ``Quantum
Symmetry'' and the Charles University PRIMUS grant \textit{Spectral
  Noncommutative Geometry of Quantum Flag Manifolds} PRIMUS/21/SCI/026. He is
grateful to Indian Statistical Institute, Kolkata and Prof. Debashish Goswami
for offering him a visiting scientist position, where this work started. He is
also grateful to SERB, Government of India, for the National postdoctoral
fellowship (PDF/2021/003544) and Indian Institute of Science Education and
Research, Kolkata for hosting him. The third author was supported by a senior
project associate fellowship at National Institute of Science Education and
Research Bhubaneswar from the INSPIRE faculty awards
(DST/INSPIRE/04/2016/000215) given by the D.S.T., Government of India. The fourth author was partially supported by a MATRICS
grant (MTR/2020/000640) given by SERB, Government of India. 

\subsection*{Notations} For two $\textup{C}^*$\nobreakdash-algebras $A$ and $B$, $A \otimes B$ denotes the minimal tensor product of $\textup{C}^*$\nobreakdash-algebras. For a $\textup{C}^*$\nobreakdash-algebra $A$ and two closed subspaces $X,Y \subset A$, $XY$ denotes the norm-closed linear span of the set of products $xy$, $x \in X$ and $y \in Y$. For an object $X$ in some category, $\mathrm{id}_X$ denotes the identity morphism of $X$. For a unital $\textup{C}^*$\nobreakdash-algebra $A$, $1_A$ denotes the unit element in $A$, $\mathcal{M}(A)$ denotes the multiplier algebra of $A$, and $\mathcal{U}(A)$ denotes the group of unitary multipliers of $A$.

\section{Anyonic compact quantum groups - definitions and examples}\label{sec:bcqg}

In this section, we briefly recall what anyonic compact quantum groups are, following \cite{R2021}, the latter being concerned with braided compact quantum groups over $\mathbb{T}$. A slightly more detailed account (of braided compact quantum groups over $\mathbb{T}$) may be found in \cite{BJR2022}, see also \cites{kmrw2016,MR2021}. For a completely general treatment, we refer the reader to \cites{mrw2014,mrw2016}. 

Let $\mathcal{C}^*$ be the category of $\textup{C}^*$\nobreakdash-algebras. For $A$ and $B$ in $\mathrm{Obj}(\mathcal{C}^*)$, we write the set of morphisms as $\mathrm{Mor}(A,B)$ which consists of nondegenerate $*$-homomorphisms $\pi : A \rightarrow \mathcal{M}(B)$, i.e., $*$-homomorphisms $\pi : A \rightarrow \mathcal{M}(B)$ such that $\pi(A)B=B$. Thus for unital $A$ and $B$, $\mathrm{Mor}(A,B)$ consists of unital $*$-homomorphisms from $A$ to $B$. To avoid confusion, we shall exclusively write a morphism as $A \rightarrow B$ or explicitly say a nondegenerate $*$\nobreakdash-homomorphism $A \rightarrow \mathcal{M}(B)$; of course there is no difference in the unital situation. We recall that a compact quantum group $G$ consists of a pair $G=(\textup{C}(G),\Delta_{G})$ where $\textup{C}(G)$ is a unital $\textup{C}^*$\nobreakdash-algebra and $\Delta_{G} : \textup{C}(G) \rightarrow \textup{C}(G) \otimes \textup{C}(G)$ is a coassociative morphism satisfying a cancellation property (\cite{Worosymm}). 

We fix a positive integer $N \in \mathbb{Z}$ and let $\omega$ denote an $N$\nobreakdash-th primitive root of unity. We write $\mathbb{Z}_N$ for the group $\mathbb{Z}/N\mathbb{Z}$. Let $z \in \textup{C}(\mathbb{Z}_N)$ denote the function that sends $t \in \mathbb{Z}_N$ to $\omega^t$, i.e., $z(t)=\omega^t$; $\textup{C}(\mathbb{Z}_N)$ is generated as a $\textup{C}^*$\nobreakdash-algebra by this unitary $z$. We also recall that the comultiplication $\Delta_{\mathbb{Z}_N} : \textup{C}(\mathbb{Z}_N) \rightarrow \textup{C}(\mathbb{Z}_N) \otimes \textup{C}(\mathbb{Z}_N)$ of the group $\mathbb{Z}_N$ sends $z$ to $z \otimes z$.

We let $\chi_s$ denote the character on $\mathbb{Z}_N$ defined by $\chi_s(t)=\omega^{st}$ for $s,t \in \mathbb{Z}_N$; thus $\widehat{\mathbb{Z}_N}=\{\chi_s \mid s \in \mathbb{Z}_N\}$. We define a bicharacter $\textup{R} \in \mathcal{U}(\textup{C}(\widehat{\mathbb{Z}_N}) \otimes \textup{C}(\widehat{\mathbb{Z}_N}))$ as follows:
\[\textup{R} : \widehat{\mathbb{Z}_N} \times \widehat{\mathbb{Z}_N} \rightarrow \mathbb{T}, \quad \textup{R}(\chi_m,\chi_n)=\omega^{mn}.\] This is an $\textup{R}$\nobreakdash-matrix in the sense of \cite{mrw2016} which we fix for the rest of the article. We also identify $\widehat{\mathbb{Z}_N}$ with $\mathbb{Z}_N$, via the bicharacter $\textup{R}$ above. This identification together with $\textup{C}^*(\mathbb{Z}_N)\cong \textup{C}(\widehat{\mathbb{Z}_N})$ yield the following:
\begin{equation}\label{eq:relation}
    \delta_{t}=\frac{1}{N} \sum_{s \in \mathbb{Z}_N}\omega^{-st}z^s \text{ and } z^s = \sum_{t \in \mathbb{Z}_N} \omega^{st}\delta_{t},
\end{equation} where $\delta_{t}$, $t \in \mathbb{Z}_N$ is the function on $\mathbb{Z}_N$ given by 
\[\delta_{t}(t')= 
\begin{cases}
    1 & t=t'; \\
    0 & \text{otherwise}.
\end{cases}\] We shall write $\Omega$ for the matrix $(\frac{1}{N}\omega^{-ij})_{i,j \in \mathbb{Z}_N}$, i.e., $\Omega_{ij}=\frac{1}{N}\omega^{-ij}$ for $i,j \in \mathbb{Z}_N$, the determinant of which is \[\det(\Omega)=\biggl(\frac{1}{N}\biggr)^{N} \prod_{t \in \mathbb{Z}_N}\frac{1}{w^{t}-w^{t- 1}} \neq 0,\] implying that $\Omega$ is invertible; furthermore the inverse is given by $\Omega^{-1}=(\omega^{ij})_{i,j \in \mathbb{Z}_N}$. 

\begin{definition}\label{def:znalg}
We define the category $\mathcal{C}^*_{\mathbb{Z}_N}$ consisting of $\mathbb{Z}_N$\nobreakdash-$\textup{C}^*$\nobreakdash-algebras and $\mathbb{Z}_N$\nobreakdash-equivariant morphisms as follows. An object of $\mathcal{C}^*_{\mathbb{Z}_N}$ is a pair $(X,\rho^X)$, where $X$ is a unital $\textup{C}^*$\nobreakdash-algebra and $\rho^X \in \mathrm{Mor}(X, X \otimes \textup{C}(\mathbb{Z}_N))$ such that
\begin{enumerate}
\item $(\rho^X \otimes \mathrm{id}_{\textup{C}(\mathbb{Z}_N)})\circ \rho^X=(\mathrm{id_X} \otimes \Delta_{\mathbb{Z}_N})\circ \rho^X$;
\item $\rho^X(X)(1_X \otimes \textup{C}(\mathbb{Z}_N))=X \otimes \textup{C}(\mathbb{Z}_N)$.
\end{enumerate}
Let $(X,\rho^X)$ and $(Y,\rho^Y)$ be two $\mathbb{Z}_N$\nobreakdash-$\textup{C}^*$\nobreakdash-algebras. A morphism $\phi : (X,\rho^X) \rightarrow (Y,\rho^Y)$ in $\mathcal{C}^*_{\mathbb{Z}_N}$ (or equivalently, a $\mathbb{Z}_N$\nobreakdash-equivariant morphism) is, by definition, $\phi \in \mathrm{Mor}(X,Y)$ such that $\rho^Y \circ \phi=(\phi \otimes \mathrm{id}_{\textup{C}(\mathbb{Z}_N)}) \circ \rho^X$. We write $\mathrm{Mor}^{\mathbb{Z}_N}(X,Y)$ for the set of morphisms between $(X,\rho^X)$ and $(Y,\rho^Y)$ in $\mathcal{C}^*_{\mathbb{Z}_N}$.
\end{definition}

\begin{remark}\label{rem:grading}
A $\mathbb{Z}_N$\nobreakdash-$\textup{C}^*$\nobreakdash-algebra $X$ (with $\rho^X$ understood) comes with an associated $\mathbb{Z}_N$-grading defined as follows. We call an element $x \in X$ homogeneous of degree $t \in \mathbb{Z}_N$ if $\rho^X(x)=x \otimes z^t$ and write $\mathrm{deg}(x)=t$. For each $t \in \mathbb{Z}_N$, we let $X(t)$ denote the set consisting of homogeneous elements of degree $t$: $X(t)=\{x \in X \mid \mathrm{deg}(x)=t\}$. The collection $\{X(t)\}_{t \in \mathbb{Z}_N}$ enjoys the following: 
\begin{enumerate}
\item for each $t \in \mathbb{Z}_N$, $X(t)$ is a closed subspace of $X$;
\item for $s, t \in \mathbb{Z}_N$, $X(s)X(t) \subseteq X([s+t])$;
\item for each $t \in \mathbb{Z}_N$, $X(t)^*=X(-t)$;
\item as a Banach space, $X$ coincides with the algebraic direct sum $\bigoplus_{t \in \mathbb{Z}_N}X(t)$.
\end{enumerate}
\end{remark}

Roughly speaking, an anyonic compact quantum group over $\mathbb{Z}_N$ or an anyonic compact quantum group as we shall call it, is a ``compact quantum group'' object in $\mathcal{C}^*_{\mathbb{Z}_N}$, which is endowed with a monoidal structure using a braided tensor product $\boxtimes_{\textup{R}}$ depending on the bicharacter $\textup{R}$ defined above. For convenience, we introduce the following notation. Given a $\mathbb{Z}_N$\nobreakdash-$\textup{C}^*$\nobreakdash-algebra $(A, \rho^A)$, $a \in A$ and $t \in \mathbb{Z}_N$, we write the value of the map $\rho^A(a) \in A \otimes \textup{C}(\mathbb{Z}_N) \cong \textup{C}(\mathbb{Z}_N,A)$ at $t$ as $\rho^A_{t}(a)$, i.e., $\rho^A(a)(t)=\rho^A_{t}(a) \in A$.

Let $(X,\rho^X)$ and $(Y,\rho^Y)$ be two objects of $\mathcal{C}^*_{\mathbb{Z}_N}$. Let $\mathcal{L}^X$ and $\mathcal{L}^Y$ be a pair of separable Hilbert spaces with continuous representations $\pi^X$ and $\pi^Y$, respectively, of $\mathbb{Z}_N$. Furthermore, let $X \hookrightarrow \mathbb{B}(\mathcal{L}^X)$ and $Y \hookrightarrow \mathbb{B}(\mathcal{L}^Y)$ be faithful, $\mathbb{Z}_N$-equivariant representations of $X$ and $Y$ on $\mathcal{L}^X$ and $\mathcal{L}^Y$, respectively. There are orthonormal bases $(\lambda^X_m)_{m \in \mathbb{N}}$ and $(\lambda^Y_m)_{m \in \mathbb{N}}$ for $\mathcal{L}^X$ and $\mathcal{L}^Y$ consisting of eigenvectors for the $\mathbb{Z}_N$-actions $\pi^X$ and $\pi^Y$, respectively, i.e., \[\pi^X(t)(\lambda^X_m)=\omega^{tl^X_m}\lambda^X_m,  \text{ for some } l^X_m \in \mathbb{N},\] and \[\pi^Y(t)(\lambda^Y_m)=\omega^{tl^Y_m}\lambda^Y_m,  \text{ for some } l^Y_m \in \mathbb{N}.\] Associated to the $\textup{R}$-matrix, we have the braiding unitaries $c_{\mathcal{L}^X,\mathcal{L}^Y} : \mathcal{L}^X \otimes \mathcal{L}^Y \rightarrow \mathcal{L}^Y \otimes \mathcal{L}^X$ and its dual $c_{\mathcal{L}^Y,\mathcal{L}^X} : \mathcal{L}^Y \otimes \mathcal{L}^X \rightarrow \mathcal{L}^X \otimes \mathcal{L}^Y$ defined by \[c_{\mathcal{L}^X,\mathcal{L}^Y}(\lambda^X_a \otimes \lambda^Y_b)=\omega^{l^X_al^Y_b}\lambda^Y_b \otimes \lambda^X_a, \quad c_{\mathcal{L}^Y,\mathcal{L}^X}(\lambda^Y_b \otimes \lambda^X_a)=\omega^{-l^X_al^Y_b}\lambda^X_a \otimes \lambda^Y_b.\] 

\begin{definition}\label{def:braided}
The braided tensor product $X \boxtimes_{\textup{R}} Y$ is defined to be the $\textup{C}^*$\nobreakdash-algebra generated by $j_1(x)j_2(y)$ in $\mathbb{B}(\mathcal{L}^X \otimes \mathcal{L}^Y)$, where $j_1(x)=x \otimes \mathrm{id}_{\mathcal{L}^Y}$ and $j_2(y)=c_{\mathcal{L}^Y,\mathcal{L}^X}(y \otimes \mathrm{id}_{\mathcal{L}^X})c_{\mathcal{L}^X,\mathcal{L}^Y}$, for $x \in X$ and $y \in Y$.      
\end{definition}

It can be shown that $X \boxtimes_{\textup{R}} Y$ is the closed linear span $j_1(X)j_2(Y)$, i.e., \[X \boxtimes_{\textup{R}} Y:=j_1(X)j_2(Y).\] By abuse of notation, we write $j_1 \in \mathrm{Mor}(X, X \boxtimes_{\textup{R}} Y)$ and $j_2 \in \mathrm{Mor}(Y, X \boxtimes_{\textup{R}} Y)$. We sometimes write $x \boxtimes_{\textup{R}}1_Y$ for $j_1(x)$ (and similarly, $1_X \boxtimes_{\textup{R}} y=j_2(y)$), so that for homogeneous $x$ and $y$,
\begin{align*}
(x \boxtimes_{\textup{R}} 1_Y)(1_X \boxtimes_{\textup{R}} y)={}&j_1(x)j_2(y)\\
={}&\omega^{\mathrm{deg}(x)\mathrm{deg}(y)}j_2(y)j_1(x)\\
={}&\omega^{\mathrm{deg}(x)\mathrm{deg}(y)}(1_X \boxtimes_{\textup{R}} y)(x \boxtimes_{\textup{R}} 1_Y).
\end{align*}

\begin{lemma}\label{lem:action}
There is a unique coaction $\rho^{X \boxtimes_{\textup{R}} Y}$ of $\textup{C}(\mathbb{Z}_N)$ on $X \boxtimes_{\textup{R}} Y$ such that $j_1 \in \mathrm{Mor}^{\mathbb{Z}_N}(X, X \boxtimes_{\textup{R}} Y)$ and $j_2 \in \mathrm{Mor}^{\mathbb{Z}_N}(Y, X \boxtimes_{\textup{R}} Y)$. 
\end{lemma}

Throughout the paper, $X \boxtimes_{\textup{R}} Y$ is equipped with this $\textup{C}(\mathbb{Z}_N)$\nobreakdash-coaction and thus becomes an object of $\mathcal{C}^*_{\mathbb{Z}_N}$. 

\begin{lemma}\label{lem:morphism}
Suppose $\pi_1 \in \mathrm{Mor}^{\mathbb{Z}_N}(X_1, Y_1)$ and $\pi_2 \in \mathrm{Mor}^{\mathbb{Z}_N}(X_2, Y_2)$ are two $\mathbb{Z}_N$\nobreakdash-equivariant morphisms. Then there exists a unique $\mathbb{Z}_N$\nobreakdash-equivariant morphism $\pi_1 \boxtimes_{\textup{R}} \pi_2 \in \mathrm{Mor}^{\mathbb{Z}_N}(X_1 \boxtimes_{\textup{R}} X_2, Y_1 \boxtimes_{\textup{R}} Y_2)$ such that \[(\pi_1 \boxtimes_{\textup{R}} \pi_2)(j_1(x_1)j_2(x_2))=j_1(\pi_1(x_1))j_2(\pi_2(x_2)),\] for $x_1 \in X_1$ and $x_2 \in X_2$.
\end{lemma}

Having gathered the required notions, we can proceed to define a braided compact quantum group over $\mathbb{Z}_N$.

\begin{definition}\label{def:bcqg}\cite{mrw2016}
A braided compact quantum group over $\mathbb{Z}_N$, or an anyonic compact quantum group, is a triple $G=(\textup{C}(G),\rho^{\textup{C}(G)}, \Delta_G)$, where $\textup{C}(G)$ is a unital $\textup{C}^*$\nobreakdash-algebra, $\rho^{\textup{C}(G)}$ is a $\textup{C}(\mathbb{Z}_N)$\nobreakdash-coaction on $\textup{C}(G)$ so that $(\textup{C}(G),\rho^{\textup{C}(G)})$ is an object of $\mathcal{C}^*_{\mathbb{Z}_N}$, $\Delta_{G}$ is a $\mathbb{Z}_N$\nobreakdash-equivariant morphism $\Delta_G \in \mathrm{Mor}^{\mathbb{Z}_N}(\textup{C}(G), \textup{C}(G) \boxtimes_{\textup{R}} \textup{C}(G))$ such that
\begin{enumerate}
\item $(\Delta_G \boxtimes_{\textup{R}} \mathrm{id}_{\textup{C}(G)}) \circ \Delta_G=(\mathrm{id}_{\textup{C}(G)} \boxtimes_{\textup{R}} \Delta_G)\circ \Delta_G$ (coassociativity);
\item $\Delta_G(\textup{C}(G))(1_{\textup{C}(G)} \boxtimes_{\textup{R}} \textup{C}(G))=\Delta_G(\textup{C}(G))(\textup{C}(G) \boxtimes_{\textup{R}} 1_{\textup{C}(G)})=\textup{C}(G) \boxtimes_{\textup{R}} \textup{C}(G)$ (bisimplifiability).
\end{enumerate}
\end{definition}

We now construct the examples of anyonic compact quantum groups we shall be concerned with in the rest of the article.

\begin{definition}\label{def:permutationalg}
We define $\textup{C}(\textup{S}_N^+(\textup{R}))$ to be the universal unital $\textup{C}^*$\nobreakdash-algebra with generators $q_{ij}$ for $i,j \in \mathbb{Z}_N$ subject to the following set of relations:
\begin{enumerate}
    \item for each $i \in \mathbb{Z}_N$, $q_{0i}=q_{i0}=\delta_{i0}$;
    \item for each $i,j \in \mathbb{Z}_N$, $q_{ij}^*=\omega^{-i(i-j)} q_{-i,-j}$;
    \item for each $i,j,k \in \mathbb{Z}_N$, $q_{k, i+j}=\sum_{l \in \mathbb{Z}_N} \omega^{-l(i-k+l)} q_{k-l,i}q_{lj}$;
    \item for each $i,j,k \in \mathbb{Z}_N$, $q_{i+j,k}=\sum_{l \in \mathbb{Z}_N} \omega^{-i(l-j)} q_{jl}q_{i,k-l}$.
\end{enumerate}
\end{definition}

\begin{remark}
We use a comma to separate the two subscripts of $q$ when the group operations of $\mathbb{Z}_N$ are applied on one or both of them, e.g., $q_{k,i+j}$; otherwise, we use the standard juxtaposition, e.g., $q_{ij}$.  
\end{remark}

To construct $\textup{C}(\textup{S}_N^+(\textup{R}))$, we first record a result (Theorem \ref{thm:quotient}) that we will prove below: the matrix $q=(q_{ij})_{i,j \in \mathbb{Z}_N} \in M_N(\textup{C}(\textup{S}_N^+(\textup{R}))$ is a unitary matrix. Granting this result, we observe that $\|q_{ij}\| \leq 1$ and the relations are polynomials in $q_{ij}$, $q^*_{ij}$, thus ensuring the existence of $\textup{C}(\textup{S}_N^+(\textup{R}))$. Now let $\mathcal{A}(\textup{S}_N^+(\textup{R}))$ be the universal unital $*$-algebra with same generators and relations. Given a $\textup{C}^*$\nobreakdash-seminorm $\|\cdot\|$ on $\mathcal{A}(\textup{S}_N^+(\textup{R}))$, we have $\|q_{ij}\| \leq 1$ for $i,j \in \mathbb{Z}_N$, hence there is a largest $\textup{C}^*$\nobreakdash-seminorm on $\mathcal{A}(\textup{S}_N^+(\textup{R}))$ and $\textup{C}(\textup{S}_N^+(\textup{R}))$ is the completion of $\mathcal{A}(\textup{S}_N^+(\textup{R}))$ in this largest $\textup{C}^*$\nobreakdash-seminorm. 

\begin{proposition}
There is a unique unital $*$-homomorphism \[\rho^{\textup{C}(\textup{S}_N^+(\textup{R}))} : \textup{C}(\textup{S}_N^+(\textup{R})) \rightarrow \textup{C}(\mathbb{Z}_N, \textup{C}(\textup{S}_N^+(\textup{R})))\] such that $\rho_{t}^{\textup{C}(\textup{S}_N^+(\textup{R}))}(q_{ij})=\omega^{t(j-i)}q_{ij}$ for each $i,j \in \mathbb{Z}_N$ and $t \in \mathbb{Z}_N$, satisfying the two conditions in Definition \textup{\ref{def:znalg}}, making $(\textup{C}(\textup{S}_N^+(\textup{R})),\rho^{\textup{C}(\textup{S}_N^+(\textup{R}))})$ a $\mathbb{Z}_N$\nobreakdash-$\textup{C}^*$\nobreakdash-algebra.
\end{proposition}

\begin{proof}
We begin by remarking that, defining $\rho_{t}^{\textup{C}(\textup{S}_N^+(\textup{R}))}(q_{ij})=\omega^{t(j-i)}q_{ij}$ yields a $\mathbb{Z}_N$-action on the free unital $*$\nobreakdash-algebra with generators $q_{ij}$. To conclude the proof, we need to show that the defining relations (see Definition \ref{def:permutationalg}) are homogeneous. The relation (1) is homogeneous because \[\rho_{t}^{\textup{C}(\textup{S}_N^+(\textup{R}))}(q_{i0})=\omega^{-ti}q_{i0}, \quad \rho_{t}^{\textup{C}(\textup{S}_N^+(\textup{R}))}(q_{0i})=\omega^{ti}q_{0i},\] for each $i$. For (2), we observe that 
\allowdisplaybreaks{
\begin{align*}
\rho_{t}^{\textup{C}(\textup{S}_N^+(\textup{R}))}(q^*_{ij})=\omega^{t(-j+i)}q^*_{ij}={}&\omega^{t(-j+i)}\omega^{-i(i-j)}q_{-i,-j}\\
={}&\rho_{t}^{\textup{C}(\textup{S}_N^+(\textup{R}))}(\omega^{-i(i-j)}q_{-i,-j}),    
\end{align*}
}for all $i,j$. The relation (3) is homogeneous because
\allowdisplaybreaks{
\begin{align*}
\rho_{t}^{\textup{C}(\textup{S}_N^+(\textup{R}))}(q_{k,i+j})={}&\omega^{t(i+j-k)}q_{k,i+j}\\
={}&\omega^{t(i+j-k)}\sum_{l \in \mathbb{Z}_N} \omega^{-l(i-k+l)} q_{k-l,i}q_{lj}\\
={}&\sum_{l \in \mathbb{Z}_N} \omega^{-l(i-k+l)} \omega^{t(i-k+l)}\omega^{t(j-l)}q_{k-l,i}q_{lj}\\
={}&\rho_{t}^{\textup{C}(\textup{S}_N^+(\textup{R}))}(\sum_{l \in \mathbb{Z}_N} \omega^{-l(i-k+l)} q_{k-l,i}q_{lj}),
\end{align*}
}for all $i,j,k$. That the relation (4) too is homogeneous can be shown in exactly similar manner and so we skip the argument.
\end{proof}

\begin{proposition}\label{prop:comult}
There is a unique unital $*$-homomorphism \[\Delta_{\textup{S}_N^+(\textup{R})} : \textup{C}(\textup{S}_N^+(\textup{R})) \rightarrow \textup{C}(\textup{S}_N^+(\textup{R})) \boxtimes_{\textup{R}} \textup{C}(\textup{S}_N^+(\textup{R}))\] such that $\Delta_{\textup{S}_N^+(\textup{R})}(q_{ij})=\sum_{k \in \mathbb{Z}_N}j_1(q_{ik})j_2(q_{kj})$ for $i,j \in \mathbb{Z}_N$. Furthermore, $\Delta_{\textup{S}_N^+(\textup{R})}$ is $\mathbb{Z}_N$\nobreakdash-equivariant, coassociative and bisimplifiable \textup{(}see Definition \textup{\ref{def:bcqg})}.
\end{proposition}

\begin{proof}
Let $Q_{ij}=\sum_{k \in \mathbb{Z}_N}j_1(q_{ik})j_2(q_{kj})$ for $i,j \in \mathbb{Z}_N$ and $Q=(Q_{ij})$. We remark that $Q_{ij}$ is homogeneous of degree $j-i$. By the universal property, a necessarily unique unital $*$\nobreakdash-homomorphism $\Delta_{\textup{S}_N^+(\textup{R})} : \textup{C}(\textup{S}_N^+(\textup{R})) \rightarrow \textup{C}(\textup{S}_N^+(\textup{R})) \boxtimes_{\textup{R}} \textup{C}(\textup{S}_N^+(\textup{R}))$ satisfying $\Delta_{\textup{S}_N^+(\textup{R})}(q_{ij})=Q_{ij}$ exists if and only if $Q_{ij}$ satisfies the relations (1)-(4) in Definition \ref{def:permutationalg}. For relation (1), we observe that 
\allowdisplaybreaks{
 \begin{align*}
Q_{i0}=\sum_{k \in \mathbb{Z}_N}j_1(q_{ik})j_2(q_{k0})={}&\sum_{k \in \mathbb{Z}_N}j_1(q_{ik})j_2(\delta_{k0})=\delta_{i0},
\end{align*}   
}and similarly, $Q_{0i}=\delta_{i0}$. Now using relation (2) for $q_{ij}$, we obtain
\allowdisplaybreaks{
\begin{align*}
Q^*_{ij}=\sum_{k \in \mathbb{Z}_N}j_2(q^*_{kj})j_1(q^*_{ik})={}&\sum_{k \in \mathbb{Z}_N}j_2(\omega^{-k(k-j)}q_{-k,-j})j_1(\omega^{-i(i-k)}q_{-i,-k})\\
={}&\sum_{k \in \mathbb{Z}_N}\omega^{-(-j+k)(-k+i)-i(i-k)-k(k-j)}j_1(q_{-i,-k})j_2(q_{-k,-j})\\    
={}&\sum_{k \in \mathbb{Z}_N}\omega^{-i(i-j)}j_1(q_{-i,-k})j_2(q_{-k,-j})\\    
={}&\omega^{-i(i-j)}Q_{-i,-j},
\end{align*}    
}where the third equality uses the commutation relation of $j_1$ and $j_2$. Next, on one hand,
\allowdisplaybreaks{
\begin{align*}
Q_{k,i+j}={}&\sum_{\alpha \in \mathbb{Z}_N}j_1(q_{k\alpha})j_2(q_{\alpha, i+j})\\
={}&\sum_{\alpha,l \in \mathbb{Z}_N}\omega^{-l(i-\alpha+l)}j_1(q_{k\alpha})j_2(q_{\alpha-l,i}q_{lj}),   
\end{align*}   
}and on the other,
\allowdisplaybreaks{
\begin{align*}
{}&\sum_{l \in \mathbb{Z}_N}\omega^{-l(i-k+l)}Q_{k-l,i}Q_{lj}\\
={}&\sum_{l,\alpha,\beta \in \mathbb{Z}_N}\omega^{-l(i-k+l)}j_1(q_{k-l,\alpha})j_2(q_{\alpha,i})j_1(q_{l\beta}))j_2(q_{\beta j})\\
={}&\sum_{l,\alpha,\beta \in \mathbb{Z}_N}\omega^{-l(i-k+l)-(i-\alpha)(\beta-l)}j_1(q_{k-l,\alpha}q_{l\beta}))j_2(q_{\alpha,i}q_{\beta j})\\        
={}&\sum_{\alpha,\beta \in \mathbb{Z}_N}\omega^{-i\beta+\alpha\beta}j_1(\sum_{l \in \mathbb{Z}_N}\omega^{-l(\alpha-k+l)}q_{k-l,\alpha}q_{l\beta}))j_2(q_{\alpha,i}q_{\beta j})\\        
={}&\sum_{\alpha,\beta \in \mathbb{Z}_N}\omega^{-i\beta+\alpha\beta}j_1(q_{k,\alpha+\beta})j_2(q_{\alpha i}q_{\beta j})\\        
={}&\sum_{\alpha,l \in \mathbb{Z}_N}\omega^{-l(-i-\alpha+l)}j_1(q_{k,\alpha})j_2(q_{\alpha -l,i}q_{l j}),
\end{align*}    
}where the second equality is obtained by commuting $j_2$ and $j_1$; the fourth equality is obtained from using relation (3) for $q_{ij}$ and the fifth equality is obtained by replacing $\beta$ with $l$ and $\alpha$ with $\alpha-l$. Therefore, we have obtained that for all $i,j,k$, $Q_{k,i+j}=\sum_{l \in \mathbb{Z}_N}\omega^{-l(i-k+l)}Q_{k-l,i}Q_{lj}$ and by a similar argument, we also have $Q_{i+j,k}=\sum_{l \in \mathbb{Z}_N} \omega^{-i(l-j)} Q_{jl}Q_{i,k-l}$. Taking all these together, we have constructed a unique unital $*$\nobreakdash-homomorphism $\Delta_{\textup{S}_N^+(\textup{R})} : \textup{C}(\textup{S}_N^+(\textup{R})) \rightarrow \textup{C}(\textup{S}_N^+(\textup{R})) \boxtimes_{\textup{R}} \textup{C}(\textup{S}_N^+(\textup{R}))$ satisfying $\Delta_{\textup{S}_N^+(\textup{R})}(q_{ij})=Q_{ij}$ for $i,j \in \mathbb{Z}_N$.

As remarked above, for $i,j \in \mathbb{Z}_N$, $Q_{ij}$ is homogeneous of degree $j-i$ and so $\Delta_{\textup{S}_{N}^{+}(\textup{R})}$ is $\mathbb{Z}_N$\nobreakdash-equivariant. Also, since both $(\Delta_{\textup{S}_{N}^{+}(\textup{R})} \boxtimes_{\textup{R}} \mathrm{id}_{\textup{C}(\textup{S}_N^+(\textup{R}))}) \circ \Delta_{\textup{S}_{N}^{+}(\textup{R})}$ and $(\mathrm{id}_{\textup{C}(\textup{S}_N^+(\textup{R}))} \boxtimes_{\textup{R}} \Delta_{\textup{S}_{N}^{+}(\textup{R})})\circ \Delta_{\textup{S}_{N}^{+}(\textup{R})}$ send $q_{ij}$ to $\sum_{k,l \in \mathbb{Z}_N}j_1(q_{ik})j_2(q_{kl})j_3(q_{lj})$, we see that $\Delta_{\textup{S}_{N}^{+}(\textup{R})}$ is coassociative. Bisimplifiability will follow using the standard argument as in the ordinary case (see for example the proof of Proposition 2.18 in \cite{BJR2022}) once we show that the matrix $q=(q_{ij})$ is a unitary, which is the content of Theorem \ref{thm:quotient} below. This finishes the proof.
\end{proof}

\begin{definition}\label{def:permutation}
We define the anyonic quantum permutation group, denoted $\textup{S}_N^+(\textup{R})$, to be the anyonic compact quantum group $(\textup{C}(\textup{S}_N^+(\textup{R})),\rho^{\textup{C}(\textup{S}_N^+(\textup{R}))},\Delta_{\textup{S}_N^+(\textup{R})})$, constructed above. 
\end{definition}

Having constructed the anyonic quantum permutation group, we now proceed to construct another example.

\begin{definition}\label{def:unitaryalg}
We define $\textup{C}(\textup{U}_N^{+}(\textup{R}))$ to be the universal unital $\textup{C}^*$\nobreakdash-algebra with generators $u_{ij}$ for $i,j \in \mathbb{Z}_N$ subject to the relations that make $u$ and $\overline{u}_{\textup{R}}$ unitaries, where $u=(u_{ij})_{i,j \in \mathbb{Z}_N}$ and $\overline{u}_{\textup{R}}=(\omega^{-i(j-i)}u^*_{ij})_{i,j \in \mathbb{Z}_N}$.
\end{definition}
    
To construct $\textup{C}(\textup{U}_N^{+}(\textup{R}))$, we first observe that $\|u_{ij}\| \leq 1$ and the relations are polynomials in $u_{ij}$, $u^*_{ij}$, thus ensuring its existence. Now let $\mathcal{A}(\textup{U}_N^{+}(\textup{R}))$ be the universal unital $*$-algebra with same generators and relations. Given a $\textup{C}^*$\nobreakdash-seminorm $\|\cdot\|$ on $\mathcal{A}(\textup{U}_N^{+}(\textup{R}))$, we have $\|u_{ij}\| \leq 1$ for $i,j \in \mathbb{Z}_N$, hence there is a largest $\textup{C}^*$\nobreakdash-seminorm on $\mathcal{A}(\textup{U}_N^{+}(\textup{R}))$ and $\textup{C}(\textup{U}_N^{+}(\textup{R}))$ is the completion of $\mathcal{A}(\textup{U}_N^{+}(\textup{R}))$ in this largest $\textup{C}^*$\nobreakdash-seminorm. 
    
\begin{proposition}
There is a unique unital $*$-homomorphism \[\rho^{\textup{C}(\textup{U}_N^{+}(\textup{R}))} : \textup{C}(\textup{U}_N^{+}(\textup{R})) \rightarrow \textup{C}(\mathbb{Z}_N, \textup{C}(\textup{U}_N^{+}(\textup{R})))\] such that $\rho_{t}^{\textup{C}(\textup{U}_N^{+}(\textup{R}))}(u_{ij})=\omega^{t(j-i)}u_{ij}$ for each $i,j \in \mathbb{Z}_N$ and $t \in \mathbb{Z}_N$, satisfying the two conditions in Definition \textup{\ref{def:znalg}}, making $(\textup{C}(\textup{U}_N^{+}(\textup{R})),\rho^{\textup{C}(\textup{U}_N^{+}(\textup{R}))})$ a $\mathbb{Z}_N$\nobreakdash-$\textup{C}^*$\nobreakdash-algebra.
\end{proposition}
    
\begin{proof}
The proof is similar to that of Proposition 2.17 of \cite{BJR2022} and so we omit it.  
\end{proof}
    
\begin{proposition}\label{prop:comultuni}
There is a unique unital $*$-homomorphism \[\Delta_{\textup{U}_N^{+}(\textup{R})} : \textup{C}(\textup{U}_N^{+}(\textup{R})) \rightarrow \textup{C}(\textup{U}_N^{+}(\textup{R})) \boxtimes_{\textup{R}} \textup{C}(\textup{U}_N^{+}(\textup{R}))\] such that $\Delta_{\textup{U}_N^{+}(\textup{R})}(u_{ij})=\sum_{k \in \mathbb{Z}_N}j_1(u_{ik})j_2(u_{kj})$ for $i,j \in \mathbb{Z}_N$. Furthermore, $\Delta_{\textup{U}_N^{+}(\textup{R})}$ is $\mathbb{Z}_N$\nobreakdash-equivariant, coassociative and bisimplifiable \textup{(}see Definition \textup{\ref{def:bcqg})}.
\end{proposition}
    
\begin{proof}
The proof is similar to that of Proposition 2.18 of \cite{BJR2022} and so we omit it.       
\end{proof}
    
\begin{definition}\label{def:unitary}
We define the anyonic free unitary quantum group, denoted $\textup{U}_N^{+}(\textup{R})$, to be the anyonic compact quantum group $(\textup{C}(\textup{U}_N^{+}(\textup{R})),\rho^{\textup{C}(\textup{U}_N^+(\textup{R}))},\Delta_{\textup{U}_N^+(\textup{R})})$, constructed above. 
\end{definition}

\begin{theorem}\label{thm:quotient}
There is a unique $\mathbb{Z}_N$\nobreakdash-equivariant Hopf $*$-homomorphism \[\phi: \textup{C}(\textup{U}_N^{+}(\textup{R})) \rightarrow \textup{C}(\textup{S}_N^{+}(\textup{R}))\] such that $\phi(u_{ij})=q_{ij}$ for $i,j \in \mathbb{Z}_N$.
\end{theorem}

\begin{proof}
We begin by remarking that by the universal property, a necessarily unique unital $*$\nobreakdash-homomorphism $\phi: \textup{C}(\textup{U}_N^{+}(\textup{R})) \rightarrow \textup{C}(\textup{S}_N^{+}(\textup{R}))$ satisfying $\phi(u_{ij})=q_{ij}$ exists if and only if the matrices $q=(q_{ij})_{i,j \in \mathbb{Z}_N}$ and $\overline{q}_{\textup{R}}=(\omega^{-i(j-i)}q^*_{ij})_{i,j \in \mathbb{Z}_N}$ are unitaries. Now, on one hand
\allowdisplaybreaks{
\begin{align*}
\sum_{k \in \mathbb{Z}_N}q^*_{ki}q_{kj}=\sum_{k \in \mathbb{Z}_N}\omega^{-k(k-i)}q_{-k,-i}q_{kj}=q_{0,j-i}=\delta_{ij},    
\end{align*}    
}and on the other,
\allowdisplaybreaks{
\begin{align*}
\sum_{k \in \mathbb{Z}_N}q_{ik}q^*_{jk}={}&\sum_{k \in \mathbb{Z}_N}\omega^{-j(j-k)}q_{ik}q_{-j,-k}\\
={}&\omega^{-j^2+ji}\sum_{k \in \mathbb{Z}_N}\omega^{j(k-i)}q_{ik}q_{-j,-k}\\
={}&\omega^{-j^2+ji}q_{i-j,0}=\omega^{-j^2+ji}\delta_{i-j,0}=\delta_{ij};    
\end{align*}   
}here we have used the relations (2), (3) and (4) from Definition \ref{def:permutationalg}. Thus the matrix $q$ is indeed a unitary. Next,
\allowdisplaybreaks{
\begin{align*}
\sum_{k \in \mathbb{Z}_N}\omega^{k(i-j)}q_{ki}q^*_{kj}={}&\sum_{k \in \mathbb{Z}_N}\omega^{k(i-j)-k(k-j)}q_{ki}q_{-k,-j}\\
={}&\sum_{k \in \mathbb{Z}_N}\omega^{-k(k-i)}q_{ki}q_{-k,-j}\\
={}&q_{0,i-j}=\delta_{ij}, 
\end{align*}   
}and 
\allowdisplaybreaks{
\begin{align*}
\sum_{k \in \mathbb{Z}_N}\omega^{i(i-k)+j(k-j)}q^*_{ik}q_{jk}={}&\sum_{k \in \mathbb{Z}_N}\omega^{i(i-k)+j(k-j)-i(i-k)}q_{-i,-k}q_{jk}\\
={}&\sum_{k \in \mathbb{Z}_N}\omega^{j(k-j)}q_{-i,-k}q_{jk}\\    
={}&\omega^{-j^2+ji}\sum_{k \in \mathbb{Z}_N}\omega^{-j(-k+i)}q_{-i,-k}q_{jk}\\
={}&\omega^{-j^2+ji}q_{j-i,0}=\omega^{-j^2+ji}\delta_{j-i,0}=\delta_{ij},    
\end{align*}   
}where we have again used relations (2), (3) and (4) from Definition \ref{def:permutationalg} above. Therefore $\overline{q}_{\textup{R}}$ is also a unitary and so we have constructed a unique unital $*$\nobreakdash-homomorphism $\phi: \textup{C}(\textup{U}_N^{+}(\textup{R})) \rightarrow \textup{C}(\textup{S}_N^{+}(\textup{R}))$ satisfying $\phi(u_{ij})=q_{ij}$. Since both $u_{ij}$ and $q_{ij}$ have homogeneous degree $j-i$, we see that $\phi$ is $\mathbb{Z}_N$\nobreakdash-equivariant. Finally, that $\phi$ is a Hopf $*$\nobreakdash-homomorphism follows from the expressions of the comultiplications $\Delta_{\textup{U}_N^{+}(\textup{R})}$ and $\Delta_{\textup{S}_N^{+}(\textup{R})}$ evaluated at $u_{ij}$ and at $q_{ij}$, respectively.
\end{proof}

\section{Anyonic quantum symmetries of finite spaces}\label{sec:symm} In this section, we come to the main result of this paper, that of anyonic symmetries of a finite space, relying on the results obtained in the previous section. We begin this section with recalling a few more definitions needed in what follows. 

\begin{definition}\label{def:action}\cite{R2021}
Let $G=(\textup{C}(G),\rho^{\textup{C}(G)},\Delta_G)$ be an anyonic compact quantum group. An action of $G$ (equivalently, a $\textup{C}(G)$\nobreakdash-coaction) on a $\mathbb{Z}_N$\nobreakdash-$\textup{C}^*$\nobreakdash-algebra $(B,\rho^B)$ is a $\mathbb{Z}_N$\nobreakdash-equivariant morphism $\eta^B \in \mathrm{Mor}^{\mathbb{Z}_N}(B, B \boxtimes_{\textup{R}} \textup{C}(G))$ such that
\begin{enumerate}
    \item $(\mathrm{id}_B \boxtimes_{\textup{R}} \Delta_G)\circ \eta^B=(\eta^B \boxtimes_{\textup{R}} \mathrm{id}_{\textup{C}(G)})\circ \eta^B$ (coassociativity);
    \item $\eta^B(B)(1_B \boxtimes_{\textup{R}} \textup{C}(G))=B \boxtimes_{\textup{R}} \textup{C}(G)$ (Podle\'s condition).
\end{enumerate}
\end{definition}

\begin{definition}\label{def:faithful}\cite{R2021}
Let $(B,\rho^B)$ be a $\mathbb{Z}_N$\nobreakdash-$\textup{C}^*$\nobreakdash-algebra equipped with a $G$\nobreakdash-action $\eta^B \in \mathrm{Mor}^{\mathbb{Z}_N}(B,B\boxtimes_{\textup{R}} \textup{C}(G))$, where $G=(\textup{C}(G),\rho^{\textup{C}(G)},\Delta_G)$ is an anyonic compact quantum group. A $\mathbb{Z}_N$\nobreakdash-equivariant state $f : B \rightarrow \mathbb{C}$ on $B$ is one that satisfies 
\[(f \otimes \mathrm{id}_{\textup{C}(\mathbb{Z}_N)})\rho^B(b)=f(b)1_{\textup{C}(\mathbb{Z}_N)} \text{ for all }b \in B.\] The action $\eta^{B}$ is said to be faithful if the $*$-algebra generated by $\{(f \boxtimes_{\textup{R}} \mathrm{id}_{\textup{C}(G)})\eta^B(B) \mid f : B \rightarrow \mathbb{C} \text{ a } \mathbb{Z}_N\text{-equivariant state}\}$ is norm-dense in $\textup{C}(G)$.
\end{definition}

Let $X_N=\{x_i \mid i \in \mathbb{Z}_N\}$ be the finite space consisting of $N$ points. Then $\textup{C}(X_N)$ is the $\textup{C}^*$\nobreakdash-algebra generated by $N$ orthogonal projections $p_i$, $i \in \mathbb{Z}_N$ such that $\sum_{i \in \mathbb{Z}_N}p_i=1$, i.e., $\textup{C}(X_N)=\textup{C}^*\{p_i \mid p_i^2=p_i=p_i^*, \ \ \sum_{j \in \mathbb{Z}_N}p_j=1, \ \ i \in \mathbb{Z}_N\}$. $\textup{C}(X_N)$ comes equipped with a natural $\mathbb{Z}_N$\nobreakdash-action $\rho^{\textup{C}(X_N)} : \textup{C}(X_N) \rightarrow \textup{C}(X_N) \otimes \textup{C}(\mathbb{Z}_N)$ given by $\rho^{\textup{C}(X_N)}(p_j)=\sum_{i \in \mathbb{Z}_N} p_{j-i} \otimes \delta_{i}$. We introduce the following elements: for each $j \in \mathbb{Z}_N$, let \[P_j=\frac{1}{N}\sum_{i \in \mathbb{Z}_N}\omega^{ij}p_i.\] It follows from Eq.\eqref{eq:relation} that besides forming a basis of $\textup{C}(X_N)$, the elements $P_i$, $i \in \mathbb{Z}_N$ are homogeneous, $\mathrm{deg}(P_i)=i$ and satisfy
\begin{equation}\label{eq:P's}
P_0=\frac{1}{N}, \quad P_i^*=P_{-i}, \quad P_iP_j=\frac{1}{N}P_{i+j}.    
\end{equation}
Collecting the above relations together we have, \[\textup{C}(X_N)=\textup{C}^*\{P_i \mid P_0=\frac{1}{N},  P^*_i=P_{-i}, P_iP_j=\frac{1}{N}P_{i+j}, i,j \in \mathbb{Z}_N\}.\]

\begin{definition}\label{def:cat}
We define the category $\mathcal{C}(X_N)$ as follows. 
\begin{enumerate}
\item An object of $\mathcal{C}(X_N)$ is a pair $(G,\eta)$, where $G=(\textup{C}(G),\rho^{\textup{C}(G)},\Delta_G)$ is an anyonic compact quantum group, and $\eta \in \mathrm{Mor}^{\mathbb{Z}_N}(\textup{C}(X_N), \textup{C}(X_N) \boxtimes_{\textup{R}} \textup{C}(G))$ is a faithful (see Definition \ref{def:faithful}) action of $G$ on $\textup{C}(X_N)$. 
\item Let $(G_1,\eta_1)$ and $(G_2,\eta_2)$ be two objects in $\mathcal{C}(X_N)$. A morphism $\phi : (G_1,\eta_1) \rightarrow (G_2,\eta_2)$ in $\mathcal{C}(X_N)$ is by definition a $\mathbb{Z}_N$\nobreakdash-equivariant Hopf $*$-homomorphism $\phi : \textup{C}(G_2) \rightarrow \textup{C}(G_1)$ such that $(\mathrm{id}_{\textup{C}(X_N)} \boxtimes_{\textup{R}} \phi)\circ \eta_2=\eta_1$.
\end{enumerate} 
\end{definition}
    
\begin{definition}
A terminal object in $\mathcal{C}(X_N)$ is called the anyonic quantum symmetry group of $X_N$ and denoted $(\mathrm{Aut}(\textup{C}(X_N)),\eta^{X_N})$. 
\end{definition}
    
A priori, it is not clear that $(\mathrm{Aut}(\textup{C}(X_N)),\eta^{X_N})$ exists but as we shall see below, it indeed does. This is the main theorem of this section and we shall step by step build up to its proof, identifying it explicitly, in the process.

\begin{proposition}\label{prop:universal}
Let $(G,\eta) \in \mathrm{Obj}(\mathcal{C}(X_N))$ be an object of the category $\mathcal{C}(X_N)$. Then there is a surjective $\mathbb{Z}_N$\nobreakdash-equivariant Hopf $*$\nobreakdash-homomorphism $\psi_G : \textup{C}(\textup{S}^+_N(\mathrm{R})) \rightarrow \textup{C}(G)$.
\end{proposition}

\begin{proof}
Since the elements $P_0,\dots,P_{N-1}$ form a basis of $\textup{C}(X_N)$, $\eta$ is completely determined by its values on $P_i$, $i \in \mathbb{Z}_N$. Therefore, we let $\eta(P_j)=\sum_{i \in \mathbb{Z}_N} j_1(P_i)j_2(a_{ij})$, where $a_{ij} \in \textup{C}(G)$ and we observe that the following relations hold in $\textup{C}(X_N) \boxtimes_{\mathrm{R}} C(G)$:
\begin{equation}
\eta(P_0)=\frac{1}{N}\eta(1), \quad \eta(P_i^*)=\eta(P_{-i}), \quad \eta(P_iP_j)=\frac{1}{N}\eta(P_{i+j}).
\end{equation}   
We write these equations in terms of $a_{ij}$, $i,j \in \mathbb{Z}_N$, for which we simply compute, using the fact that $\eta$ is a unital $*$\nobreakdash-homomorphism. Before starting to compute, we observe that since each $P_i$ is homogeneous of degree $i$ and $\eta$ is $\mathbb{Z}_N$\nobreakdash-equivariant, $a_{ij}$ is homogeneous of degree $j-i$. Now for the first relation,
\allowdisplaybreaks{
\begin{align*}
    \eta(P_0)=\sum_{i \in \mathbb{Z}_N}j_1(P_i)j_2(a_{i0})
\end{align*}
}must equal 
\allowdisplaybreaks{
\begin{align*}
    \frac{1}{N}j_1(1)j_2(1)=j_1(P_0)j_2(1),
\end{align*}    
}which implies $a_{i0}=\delta_{i0}$, $i \in \mathbb{Z}_N$. The left-hand side of the second relation reads
\allowdisplaybreaks{
\begin{align*}
    \eta(P_j^*)={}&\sum_{i \in \mathbb{Z}_N}j_2(a^*_{ij})j_1(P_i^*)\\
    ={}&\sum_{i \in \mathbb{Z}_N}\omega^{-(j-i)i}j_1(P^*_i)j_2(a^*_{ij}),
\end{align*}    
}whereas, the right-hand side reads
\allowdisplaybreaks{
\begin{align*}
    \eta(P_{-j})={}&\sum_{i \in \mathbb{Z}_N}j_1(P_{-i})j_2(a_{-i,-j})\\     
    ={}&\sum_{i \in \mathbb{Z}_N}j_1(P^*_{i})j_2(a_{-i,-j}),    
\end{align*}   
}which upon equating with the left-hand side yields $a^*_{ij}=\omega^{-i(i-j)}a_{-i,-j}$, for $i,j \in \mathbb{Z}_N$. Finally, for the third relation, we have on one hand,
\allowdisplaybreaks{
\begin{align*}
    \eta(P_iP_j)={}&\sum_{k,l \in \mathbb{Z}_N}j_1(P_{k})j_2(a_{k i})j_1(P_{l})j_2(a_{lj})\\
    ={}&\sum_{k,l \in \mathbb{Z}_N}\omega^{-(i-k)l}j_1(P_{k}P_{l})j_2(a_{k i}a_{lj})\\
    ={}&\sum_{k,l \in \mathbb{Z}_N}\omega^{-(i-k)l}\frac{1}{N}j_1(P_{k+l})j_2(a_{k i}a_{lj})\\
    ={}&\sum_{\alpha,l \in \mathbb{Z}_N}\omega^{-(i-\alpha+l)l}\frac{1}{N}j_1(P_{\alpha})j_2(a_{\alpha-l,i}a_{lj})\\
    ={}&\sum_{\alpha \in \mathbb{Z}_N}j_1(P_{\alpha})j_2\bigl(\sum_{l \in \mathbb{Z}_N}\omega^{-(i-\alpha+l)l}\frac{1}{N}a_{\alpha-l,i}a_{lj}\bigr),
\end{align*}    
}(here, in the third equality, we have replaced $k+l$ by $\alpha$ and $k$ by $\alpha-l$) and on the other, 
\allowdisplaybreaks{
\begin{align*}
    \frac{1}{N}\eta(P_{i+j})={}&\frac{1}{N}\sum_{\alpha \in \mathbb{Z}_N}j_1(P_{\alpha})j_2(a_{\alpha,i+j})\\
    ={}&\sum_{\alpha \in \mathbb{Z}_N}j_1(P_{\alpha})j_2(\frac{1}{N}a_{\alpha,i+j});
\end{align*}    
}upon equating the two, we obtain $\sum_{l \in \mathbb{Z}_N}\omega^{-(i-\alpha+l)l}\frac{1}{N}a_{\alpha-l,i}a_{lj}=\frac{1}{N}a_{\alpha,i+j}$, i.e., $\sum_{l \in \mathbb{Z}_N}\omega^{-(i-\alpha+l)l}a_{\alpha-l,i}a_{lj}=a_{\alpha,i+j}$, $i,j,\alpha \in \mathbb{Z}_N$. The last relation still holds after passing to the bosonization $G \rtimes \mathbb{Z}_N$, (see Appendix \ref{appendix:boso}). Since $G \rtimes \mathbb{Z}_N$ is a compact quantum group, we obtain the other half, namely, $\sum_{l \in \mathbb{Z}_N}\omega^{-i(l-j)}a_{jl}a_{i,\alpha-l}=a_{i+j,\alpha}$, $i,j,\alpha \in \mathbb{Z}_N$. 

Collecting all the relations together, we obtain that $a_{ij}$ for $i,j \in \mathbb{Z}_N$ satisfy the relations (1)-(4) in Definition \ref{def:permutationalg}, and so by universality of the $\textup{C}^*$\nobreakdash-algebra $\textup{C}(\textup{S}^+_N(\mathrm{R}))$, there is a unique unital $*$\nobreakdash-homomorphism $\psi_G : \textup{C}(\textup{S}^+_N(\mathrm{R})) \rightarrow \textup{C}(G)$ such that $\psi_G(q_{ij})=a_{ij}$ for each $i,j \in \mathbb{Z}_N$. As remarked above each $a_{ij}$ is homogeneous of degree $j-i$, so $\psi_G$ is $\mathbb{Z}_N$\nobreakdash-equivariant. That $\psi_G$ is a Hopf $*$\nobreakdash-homomorphism follows from the fact that $\eta$ is coassociative and $\psi_G$ is surjective because of faithfulness of $\eta$, yielding all the requirements of $\psi_G$ and thus completing the proof.
\end{proof}

\begin{proposition}\label{prop:permobject}
There is a unique unital $*$\nobreakdash-homomorphism $\eta^{\textup{C}(X_N)} : \textup{C}(X_N) \rightarrow \textup{C}(X_N) \boxtimes_{\mathrm{R}} \textup{C}(\textup{S}^+_N(\mathrm{R}))$ such that $\eta^{\textup{C}(X_N)}(P_j)=\sum_{i \in \mathbb{Z}_N}j_1(P_i)j_2(q_{ij})$ for $i,j \in \mathbb{Z}_N$. Furthermore, $\eta^{\textup{C}(X_N)}$ is $\mathbb{Z}_N$\nobreakdash-equivariant, coassociative and satisfies Podle\'s condition \textup{(}see Definition \textup{\ref{def:action})}.
\end{proposition}

\begin{proof}
Let $P'_j=\sum_{i \in \mathbb{Z}_N}j_1(P_i)j_2(q_{ij})$ for $i,j \in \mathbb{Z}_N$. We remark that each $P'_j$ is homogeneous of degree $j$. Now, by the universal property, we see that a (necessarily unique) $*$\nobreakdash-homomorphism $\eta^{\textup{C}(X_N)}$ satisfying $\eta^{\textup{C}(X_N)}(P_j)=P'_j$ exists if and only if $P'_j$ satisfy $P'_0=\frac{1}{N}$, $P'^*_i=P'_{-i}$ and $P'_iP'_j=\frac{1}{N}P'_{i+j}$, for $i,j \in \mathbb{Z}_N$. To see that this is indeed the case, we again compute. For the first relation, we observe that
\allowdisplaybreaks{
\begin{align*}
    P'_0={}&\sum_{i \in \mathbb{Z}_N}j_1(P_i)j_2(q_{i0})\\
    ={}&\sum_{i \in \mathbb{Z}_N}j_1(P_i)j_2(\delta_{i0})=\frac{1}{N},
\end{align*}
}where we have used relation (1) of Definition \ref{def:permutationalg}. Next,
\allowdisplaybreaks{
\begin{align*}
    P'^*_j={}&\sum_{i \in \mathbb{Z}_N}j_2(q^*_{ij})j_1(P^*_i)\\
    ={}&\sum_{i \in \mathbb{Z}_N}\omega^{i(i-j)}j_1(P^*_i)j_2(q^*_{ij})\\
    ={}&\sum_{i \in \mathbb{Z}_N}\omega^{i(i-j)}j_1(P_{-i})\omega^{-i(i-j)}j_2(q_{-i,-j})\\
    ={}&\sum_{i \in \mathbb{Z}_N}j_1(P_{-i})j_2(q_{-i,-j})=P'_{-j},
\end{align*}    
}where in the third equality we have used the relation $P^*_i=P_{-i}$ and the relation (2) of Definition \ref{def:permutationalg}. Finally, for the third relation,
\allowdisplaybreaks{
\begin{align*}
    P'_iP'_j={}&\sum_{\alpha,\beta \in \mathbb{Z}_N}j_1(P_{\alpha})j_2(q_{\alpha i})j_1(P_{\beta})j_2(q_{\beta j})\\
    ={}&\sum_{\alpha,\beta \in \mathbb{Z}_N}\omega^{-(i-\alpha)\beta}j_1(P_{\alpha}P_{\beta})j_2(q_{\alpha i}q_{\beta j})\\
    ={}&\sum_{\alpha,\beta \in \mathbb{Z}_N}\omega^{-(i-\alpha)\beta}\frac{1}{N}j_1(P_{\alpha+\beta})j_2(q_{\alpha i}q_{\beta j})\\
    ={}&\sum_{\beta,k \in \mathbb{Z}_N}\omega^{-(i-k+\beta)\beta}\frac{1}{N}j_1(P_{k})j_2(q_{k-\beta, i}q_{\beta j})\\
    ={}&\frac{1}{N}\sum_{k \in \mathbb{Z}_N}j_1(P_{k})j_2\bigl(\sum_{\beta=0}^{N-1}\omega^{-(i-k+\beta)\beta}q_{k-\beta, i}q_{\beta j}\bigr)\\
    ={}&\frac{1}{N}\sum_{k \in \mathbb{Z}_N}j_1(P_{k})j_2(q_{k,i+j})\\
    ={}&\frac{1}{N}P'_{i+j},
\end{align*}    
}where, in the third equality, we have used the relation $P_{\alpha}P_{\beta}=\frac{1}{N}P_{\alpha+\beta}$; in the fourth equality, we have replace $\alpha+\beta$ by $k$ and $\alpha$ by $k-\beta$; in the sixth equality, we have used relation (3) of Definition \ref{def:permutationalg}. Therefore, we have constructed a unique and unital $*$\nobreakdash-homomorphism $\eta^{\textup{C}(X_N)} : \textup{C}(X_N) \rightarrow \textup{C}(X_N) \boxtimes_{\mathrm{R}} \textup{C}(\textup{S}^+_N(\mathrm{R}))$ satisfying $\eta^{\textup{C}(X_N)}(P_j)=P'_j$ for $j \in \mathbb{Z}_N$. 

As remarked above, for $j \in \mathbb{Z}_N$, $P'_j$ is homogeneous of degree $j$ and so $\eta^{\textup{C}(X_N)}$ is $\mathbb{Z}_N$\nobreakdash-equivariant. The coassociativity and the Podle\'s condition can be proved along the same lines as in the proof of Proposition 2.18 in \cite{BJR2022}.
\end{proof}

\begin{corollary}
The pair $(\textup{S}^+_N(\mathrm{R}),\eta^{\textup{C}(X_N)})$ is an object of the category $\mathcal{C}(X_N)$.
\end{corollary}

\begin{proof}
The result follows from the above proposition.
\end{proof}

\begin{corollary}\label{cor:intertwine}
Let $\psi_G : \textup{C}(\textup{S}^+_N(\mathrm{R})) \rightarrow \textup{C}(G)$ be the surjective $\mathbb{Z}_N$\nobreakdash-equivariant Hopf $*$\nobreakdash-homomorphism from Proposition \textup{\ref{prop:universal}}. Then it induces a morphism, again denoted by $\psi_G$, $\psi_G : (G,\eta) \rightarrow (\textup{S}^+_N(\mathrm{R}),\eta^{\textup{C}(X_N)})$ in the category $\mathcal{C}(X_N)$.      
\end{corollary}

\begin{proof}
The result follows from the explicit forms of $\psi_G$ and $\eta^{\textup{C}(X_N)}$.
\end{proof}

\begin{theorem}\label{thm:main}
The pair $(\textup{S}^+_N(\mathrm{R}),\eta^{\textup{C}(X_N)})$ is the terminal object of the category $\mathcal{C}(X_N)$, i.e., $(\textup{S}^+_N(\mathrm{R}),\eta^{\textup{C}(X_N)}) \cong (\mathrm{Aut}(\textup{C}(X_N)),\eta^{X_N})$.
\end{theorem}

\begin{proof}
The result follows from the two corollaries above.
\end{proof}

\begin{proposition}\label{prop:lowercases}
The underlying $\textup{C}^*$\nobreakdash-algebra $\textup{C}(\textup{S}^+_3(\mathrm{R}))$ is commutative and for $N \geq 4$, $\textup{C}(\textup{S}^+_N(\mathrm{R}))$ is noncommutative. 
\end{proposition}

Before proving the proposition, we prove the following two lemmas which will be needed in the proof of the proposition.

\begin{lemma}\label{lem:magic}
Suppose $A$ is a unital $\textup{C}^*$\nobreakdash-algebra. A matrix $u=(u_{ij})_{i,j \in \mathbb{Z}_N} 
\in M_N(A)$ is a magic unitary\textup{:}
\[
  u_{ij}^{2}=u_{ij}=u_{ij}^{*}, 
  \qquad 
  \sum_{i\in\mathbb{Z}_{N}}u_{ij}=1=\sum_{j\in\mathbb{Z}_{N}}u_{ij},
  \qquad 
  \text{for all $i,j\in\mathbb{Z}_{N}$},
\]
 if and only if $a=\Omega^{-1} u \Omega\in M_N(A)$ satisfies 
\begin{equation}\label{eq:rel-ord}
\begin{split}
a_{0i}=a_{i0}={}&\delta_{i0},\\
a_{ij}^*={}&a_{-i,-j},\\ 
\sum_{l \in \mathbb{Z}_N}a_{k-l,i}a_{lj}={}&a_{k, i+j},\\
\sum_{l \in \mathbb{Z}_N} a_{jl}a_{i,k-l}={}&a_{ i+j,k} \text{ for all }  i, j, k \in \mathbb{Z}_N.
\end{split}
\end{equation}
\end{lemma}

\begin{proof}
For $i,j \in \mathbb{Z}_N$, we have $a_{ij}=\frac{1}{N}\sum_{r,s \in \mathbb{Z}_N}\omega^{ir-sj}u_{rs}$, and therefore, the left-hand side of Eq. \eqref{eq:rel-ord} reads
\begin{equation}
    \begin{split}
        a_{i0}={}&\frac{1}{N}\sum_{r \in \mathbb{Z}_N}\omega^{ir}\sum_{s \in\mathbb{Z}_N}u_{rs},\\
        a^*_{ij}={}&\frac{1}{N}\sum_{r,s \in \mathbb{Z}_N} \omega^{-ir+sj}u^*_{rs},\\
        \sum_{l \in \mathbb{Z}_N}a_{k-l,i}a_{lj}={}&\frac{1}{N}\sum_{r,s,s' \in \mathbb{Z}_N}\omega^{kr-si-s'j}u_{rs}u_{rs'},\\
        \sum_{l \in \mathbb{Z}_N} a_{jl}a_{i,k-l}={}&\frac{1}{N}\sum_{r,r',s \in \mathbb{Z}_N}\omega^{jr+ir'-s'k}u_{rs}u_{r's}.
    \end{split}
\end{equation}Since $u$ is magic, we obtain Eq. \eqref{eq:rel-ord}. 

For $i,j \in \mathbb{Z}_N$, we have $u_{ij}=\frac{1}{N}\sum_{r,s \in \mathbb{Z}_N}\omega^{-ir+sj}a_{rs}$; we have, furthermore,
\begin{equation}
    \begin{split}
        u^*_{ij}={}&\frac{1}{N}\sum_{r,s \in \mathbb{Z}_N} \omega^{ir-sj}a^*_{rs},\\
        u^2_{ij}={}&\frac{1}{N^2}\sum_{r,s,l,l' \in \mathbb{Z}_N}\omega^{-il+jl'}a_{rs}a_{l-r,l'-s},\\
        \sum_{i \in \mathbb{Z}_N}u_{ij}={}&\frac{1}{N}\sum_{s \in \mathbb{Z}_N}\omega^{sj}a_{0s},\\
        \sum_{j \in \mathbb{Z}_N}u_{ij}={}&\frac{1}{N}\sum_{r \in \mathbb{Z}_N}\omega^{-ir}a_{r0}.
    \end{split}
\end{equation}Since $a$ satisfies Eq. \eqref{eq:rel-ord}, we obtain that $u$ is magic. This completes the proof.
\end{proof}

\begin{remark}
We observe that $a$ satisfies Eq. \eqref{eq:rel-ord} if and only if
$\overline{a}$ also satisfies Eq. \eqref{eq:rel-ord}. Now, $u$ being magic
implies that $u=\overline{u}$ and observing $\Omega^{-1}=N \overline{\Omega}$ yields  $N(\Omega u \overline{\Omega})$ satisfies Eq. \eqref{eq:rel-ord} if and only if $N(\overline{\Omega}u\Omega)$ satisfies Eq. \eqref{eq:rel-ord}.
\end{remark}

For the next lemma, suppose $v$ is a unitary generator of $\textup{C}^{*}(\mathbb{Z}_{N})$. We consider the elements $v_{1}=j_{1}(v)$ and $v_{2}=j_{2}(v)$ in $\textup{C}^{*}(\mathbb{Z}_{N})\boxtimes_{R}\textup{C}^{*}(\mathbb{Z}_{N})$ so that $v_{1}v_{2}=\omega v_{2}v_{1}$ holds.

\begin{lemma}\label{lem:unbrd-brd}
Suppose $A$ is a unital $\textup{C}^*$\nobreakdash-algebra and $a=(a_{ij})_{i,j \in \mathbb{Z}_{N}} \in M_{N}(A)$ satisfies \textup{Eq.} \eqref{eq:rel-ord}. Then $q=(q_{ij})_{i,j \in \mathbb{Z}_{N}}$, where $q_{ij}=v_{1}^{-i}v_{2}^{j-i}\otimes a_{ij} \in \textup{C}^{*}(\mathbb{Z}_{N})\boxtimes_{R}\textup{C}^{*}(\mathbb{Z}_{N})\otimes A$, satisfies the relations in Definition \textup{\ref{def:permutationalg}}.
\end{lemma}

\begin{proof}
  We begin by observing that the commutation relation between $v_{1}$ and $v_{2}$ together with the first three relations in Eq. \eqref{eq:rel-ord} imply the first three relations in Definition \textup{\ref{def:permutationalg}}, respectively. Furthermore, the following computation shows that the fourth relation holds as well.
  \allowdisplaybreaks{
\begin{align*}
\sum_{l \in \mathbb{Z}_N} \omega^{-l(i-k+l)} q_{k-l,i}q_{lj}
 &{}=\sum_{l \in \mathbb{Z}_N} \omega^{-l(i-k+l)} v_{1}^{-k+l}v_{2}^{i-k+l}v_{1}^{-l}v_{2}^{j-l}\otimes a_{k-l,i}a_{lj}\\
 &{}= v_{1}^{-k}v_{2}^{i+j-k}\otimes \sum_{l \in \mathbb{Z}_N} a_{k-l,i}a_{lj}\\
 &{}=v_{1}^{-k}v_{2}^{i+j-k}\otimes a_{k,i+j}=q_{k,i+j}.
\end{align*}
  }
\end{proof}

\begin{proof}[Proof of Proposition \ref{prop:lowercases}]
Let $p\neq 0,1$ be a projection and let us consider the matrix
\[
u=\begin{pmatrix}
1 & 0 & 0\\
0 & p & 1-p \\
0 & 1-p & p
\end{pmatrix},
\]
which is a magic unitary. Then Lemma \ref{lem:magic} implies that the matrix $a=\Omega^{-1}u\Omega$ satisfies the relations Eq. \eqref{eq:rel-ord}. In particular, we observe that $a_{11}=a_{22}=p$, $a_{12}=a_{21}=1-p$, and Lemma \ref{lem:unbrd-brd} ensures that $q_{11},q_{12},q_{21}$, and $q_{22}$ are nonzero and $q=(q_{ij})_{i,j\in\mathbb{Z}_{3}}$ satisfies the relations in Definition \ref{def:permutationalg}.

Now we consider the following relations obtained from the ones in Definition \ref{def:permutationalg}: 
\allowdisplaybreaks{
\begin{align*}
{}&q_{00}=1=q_{0,1+1}=\sum_{l \in \mathbb{Z}_3} \omega^{-l(1-0+l)}q_{3-l,1}q_{l2}=\omega q_{21}q_{12}+q_{11}q_{22},\\
{}&q_{20}=0=q_{2,1+2}=\sum_{l \in \mathbb{Z}_3} \omega^{-l(1-2+l)}q_{2-l,1}q_{l2}=q_{11}q_{12},\\
{}&q_{11}^{*}=q_{22}=q_{2,1+1}=\sum_{l \in \mathbb{Z}_3} \omega^{-l(1-2+l)}q_{2-l,1}q_{l,1}=q_{11}^{2},\\
{}&q_{12}^{*}=\omega q_{21}=\omega q_{2,2+2}= \omega \sum_{l \in \mathbb{Z}_3} \omega^{-l(2-2+l)}q_{2-l,2}q_{l2}=\omega^3 q_{12}^{2}=q_{12}^{2}.
\end{align*}   
}Consequently, $(q_{12}q_{11})^{*}=q_{11}^{*}q_{12}^{*}=q_{11}^{2}q_{12}^{2}=0$ implies $q_{22}q_{21}=0$. Thus 
$\textup{C}(\textup{S}^+_3(\mathrm{R}))$ is the universal $\textup{C}^*$\nobreakdash-algebra generated $q_{11}$ and $q_{12}$ subject to the relations $q_{11}^{*}=q_{11}^{2}$, $q_{12}^{*}=q_{12}^{2}$, $q_{11}q_{12}=0=q_{12}q_{11}$ and $q_{11}^{3}+q_{12}^{3}=1$. The relations do not depend on $\omega$, so they are the same as for the ordinary quantum permutation group on $3$ points. Hence $\textup{C}(\textup{S}^+_3(\mathrm{R}))$ is commutative and it is isomorphic to $\textup{C}(\textup{S}_3)$.

For the other part, let us choose two non commuting projections $p$ and $q$ such that $\textup{C}^{*}(p,q)$ is infinite dimensional. We set  
\[
u= \left( \begin{array}{cc}
  p & 1-p \\
  1-p & p 
  \end{array} \right) \oplus 
  \left( \begin{array}{cc}
  q & 1-q \\
  1-q & q 
  \end{array} \right) \oplus 
  I_{N-4}. 
\]Then since $u$ is an \(N\times N\) magic unitary, Lemma \ref{lem:magic} implies that $\widetilde{u}=\Omega^{-1}u\Omega$ satisfies Eq. \eqref{eq:rel-ord}. For $i,j \in \mathbb{Z}_N$, we have
\allowdisplaybreaks{
\begin{align*}
\widetilde{u}_{ij}={}& \frac{1}{N}\sum_{k,l \in \mathbb{Z}_N} \omega^{ik}u_{kl}\omega^{-lj}\\
={}&\frac{1}{N}\biggl(\sum_{k,l=0}^{3} \omega^{ik-lj}u_{kl}\biggr)
+\biggl(\frac{1}{N}\sum_{k=4}^{N-1}\omega^{k(i-j)}\biggr)\\
={}&\frac{1}{N}\biggl(p+\omega^{-j}(1-p)+\omega^{i}(1-p)+\omega^{i-j}p\biggr)\\
+{}&\frac{1}{N}\biggl(\omega^{2(i-j)}q+\omega^{2i-3j}(1-q)+\omega^{-3i+2j}(1-q) +\omega^{3(i-j)}q\biggr)\\
+{}&\biggl(\frac{1}{N}\sum_{k=4}^{N-1}\omega^{k(i-j)}\biggr)1\\
={}&\frac{1}{N}\biggl((p-1)+\omega^{-j}(1-p)+\omega^{i}(1-p)+\omega^{i-j}(p-1)\biggr)\\
+{}&\frac{1}{N}\biggl(\omega^{2(i-j)}(q-1)+\omega^{2i-3j}(1-q)+
\omega^{-3i+2j}(1-q) +\omega^{3(i-j)}(q-1)\biggr)\\
+{}&\biggl(\frac{1}{N}\sum_{k \in \mathbb{Z}_N}\omega^{k(i-j)}\biggr)1\\
={}&\frac{1}{N}\biggl((1-\omega^{-j})(1-\omega^{i})(p-1)
+ \omega^{2(i-j)}(1-\omega^{-j})(1-\omega^{i})(q-1)\biggr)
+\delta_{ij}1\\
={}&\frac{1}{N}(1-\omega^{-j})(1-\omega^{i})((p-1)
+\omega^{2(i-j)}(q-1))
+\delta_{ij}1.            
\end{align*}
}By Lemma \ref{lem:unbrd-brd} the elements $q_{ij}=v_{1}^{-i}v_{2}^{j-i}\otimes \widetilde{u}_{ij}$ satisfy the relations in Definition \ref{def:permutationalg}. Now the commutator $[\widetilde{u}_{12},\widetilde{u}_{23}]$ vanishes and, consequently the commutator $[(v_{1}\otimes 1)q_{12},(v_{1}^{2}\otimes  1)q_{23}]$ vanishes too and this gives the following relation:
\[
 (v_{1}\otimes 1)q_{12}(v_{1}^{2}\otimes 1)q_{23}
 =(v_{1}^{2}\otimes 1)q_{23}(v_{1}\otimes 1)q_{12},
\]
which is equivalent to 
\[
 \omega^{-2}q_{12}q_{23}=\omega^{-1}q_{23}q_{12} 
 \iff
 q_{12}q_{23}=\omega q_{23}q_{12}.
\]Therefore, $\textup{C}(S_{N}^{+}(\textup{R}))$ is noncommutative and infinite dimensional whenever $N\geq 4$.
\end{proof}

\begin{remark}\label{rem:outer-NC}
Although the $\textup{C}^*$\nobreakdash-algebra $\textup{C}(\textup{S}^+_3(\textup{R}))$ is commutative, the comultiplication still takes values in the braided tensor product $\textup{C}(\textup{S}^+_3(\textup{R})) \boxtimes_{\mathrm{R}} \textup{C}(\textup{S}^+_3(\textup{R}))$. This can be seen as follows:
\[\Delta_{\textup{S}^+_3(\textup{R})}(q_{11})=j_1(q_{11})j_2(q_{11})+j_1(q_{12})j_2(q_{21});\] recalling that $q_{12}$ is homogeneous of degree $1$ and $q_{21}$ is homogeneous of degree $-1$, we obtain $j_1(q_{12})j_2(q_{21})=\omega^{-1} j_2(q_{21})j_1(q_{12})$.
\end{remark}

\section{Applications to braided quantum symmetries of graphs}\label{sec:graph-symmetry}
We end this article by describing an application of our results to braided quantum symmetries of a class of graphs, following \cites{banicagraphs}. To that end, let $\Gamma=(E,V)$ be a finite, simple, undirected graph with $N$ vertices, such that $\mathbb{Z}_{N}$ is a subgroup of the ordinary automorphism group $\mathrm{Aut}(\Gamma)$ of the graph $\Gamma$. Such a graph is usually called a circulant graph but note that in \cite[Definition 4.1]{banicagraphs}, it is called a cyclic graph. By relabeling the vertices if necessary, we can assume that the $\mathbb{Z}_N$\nobreakdash-action on $\textup{C}(V)$, $\rho^{\textup{C}(V)} : \textup{C}(V) \rightarrow \textup{C}(V) \otimes \textup{C}(\mathbb{Z}_N)$ is given by the corepresentaion matrix $\Omega_{V}=(\delta_{j-i})_{i,j \in \mathbb{Z}_N}$; implicit here is the identification of $V$ with $X_{N}$, so that all the results obtained in the previous section could be brought forth. Let us denote the adjacency matrix of $\Gamma$ by $A_{\Gamma}$ which has a particularly simple form; let $P=(\delta_{i+1,j})_{i,j \in \mathbb{Z}_{N}}$ denote the cyclic permutation matrix, i.e., $$ P=
\begin{pmatrix}
  0 & \dots & 0 & 1\\
  1 & \dots & 0 & 0\\
  \vdots & \vdots & \vdots & \vdots \\
  0 & \dots & 0 & 0\\
  0 & \dots & 1 & 0
\end{pmatrix},
$$then $A_{\Gamma}=\sum_{i \in \mathbb{Z}_{N}}a_{i}P^{i}$, where $a_{0}=0$, $a_{i}=a_{-i}$ $i \in \mathbb{Z}_{N}, i \neq 0$. Since $\mathbb{Z}_{N}$ acts as automorphisms of $\Gamma$, $\Omega_{V}$ commutes with $A_{\Gamma}$, i.e., $\Omega_{V}A_{\Gamma}=A_{\Gamma}\Omega_{V}$. Let us also recall the matrix $\Omega$ introduced just before Definition \ref{def:znalg}. A simple computation shows that the matrices $\Omega P \Omega^{-1}$ and $\Omega A_{\Gamma} \Omega^{-1}$ are diagonal, with $ij$-th entry $\omega^{j}\delta_{ij}$, and $\sum_{k \in \mathbb{Z}_{N}}a_{k}\omega^{ik}\delta_{ij}$, $i,j \in \mathbb{Z}_{N}$, respectively. The latter one is nothing but $Q(\omega^{i})\delta_{ij}$ where $Q$ is the polynomial defined in \cite{banicagraphs}*{Definition 4.2}.

We use the notations introduced in the beginning of this section, just before Definition \ref{def:cat} and in the first few lines of the proof of Proposition \ref{prop:universal}. Thus an action $\eta$ of an anyonic compact quantum group is completely determined by its corepresentaion matrix $\Omega_{\eta}$. We recall $q=(q_{ij})_{i,j \in \mathbb{Z}_N} \in M_N(\textup{C}(\textup{S}^+_N(\textup{R})))$; by virtue of Proposition \ref{prop:permobject}, $q$ is also $\Omega_{\eta^{\textup{C}(X_N)}}$ - the corepresentaion matrix of the action $\eta^{\textup{C}(X_N)}$ of $\textup{S}^+_N(\textup{R})$ on $\textup{C}(X_N)$. 

\begin{definition}
Let $\eta \in \mathrm{Mor}^{\mathbb{Z}_N}(\textup{C}(V), \textup{C}(V) \boxtimes_{\textup{R}} \textup{C}(G))$ be an action of an anyonic compact quantum group $G=(\textup{C}(G),\rho^{\textup{C}(G)},\Delta_G)$ on $\textup{C}(V)$. Then $G$ is said to act by preserving the quantum symmetry of $\Gamma$ if $\Omega_{\eta}$ commutes with $\Omega A_{\Gamma} \Omega^{-1}$.
\end{definition}

\begin{definition}\label{def:graphcat}
We define the category $\mathcal{C}(\Gamma)$ as follows. 
\begin{enumerate}
\item An object of $\mathcal{C}(\Gamma)$ is a pair $(G,\eta)$, where $G=(\textup{C}(G),\rho^{\textup{C}(G)},\Delta_G)$ is an anyonic compact quantum group, and $\eta \in \mathrm{Mor}^{\mathbb{Z}_N}(\textup{C}(V), \textup{C}(V) \boxtimes_{\textup{R}} \textup{C}(G))$ is a faithful (see Definition \ref{def:faithful}) action of $G$ on $\textup{C}(V)$, acting by preserving the quantum symmetry of $\Gamma$. 
\item Let $(G_1,\eta_1)$ and $(G_2,\eta_2)$ be two objects in $\mathcal{C}(\Gamma)$. A morphism $\phi : (G_1,\eta_1) \rightarrow (G_2,\eta_2)$ in $\mathcal{C}(\Gamma)$ is by definition a $\mathbb{Z}_N$\nobreakdash-equivariant Hopf $*$-homomorphism $\phi : \textup{C}(G_2) \rightarrow \textup{C}(G_1)$ such that $(\mathrm{id}_{\textup{C}(V)} \boxtimes_{\textup{R}} \phi)\circ \eta_2=\eta_1$.
\end{enumerate} 
\end{definition}
        
\begin{definition}
A terminal object in $\mathcal{C}(\Gamma)$ is called the anyonic quantum
symmetry group of $\Gamma$ and denoted by $(\mathrm{Qaut}(\Gamma),\eta^{\Gamma})$. 
\end{definition}

We now build another anyonic compact quantum group which we will identify with $\mathrm{Qaut}(\Gamma)$. To that end, let us denote the quotient $\textup{C}(\textup{S}^+_N(\textup{R}))/[q,\Omega A_{\Gamma}\Omega^{-1}]$ of the $\textup{C}^*$\nobreakdash-algebra $\textup{C}(\textup{S}^+_N(\textup{R}))$ by $\mathcal{U}$. We shall write $[q_{ij}]$ for the class of $q_{ij}$, $i,j \in \mathbb{Z}_N$. We observe that the $ij$\nobreakdash-th entry of $q\Omega A_{\Gamma}\Omega^{-1}$ is $q_{ij}(\Omega A_{\Gamma}\Omega^{-1})_{jj}$ and the $ij$\nobreakdash-th entry of $\Omega A_{\Gamma}\Omega^{-1}q$ is $(\Omega A_{\Gamma}\Omega^{-1})_{ii}q_{ij}$, for all $i,j \in \mathbb{Z}_N$. Using this observation, the following two propositions can easily be derived and we omit the proofs therefore.

\begin{proposition}
There is a unique unital $*$-homomorphism \[\rho^{\mathcal{U}} : \mathcal{U} \rightarrow \textup{C}(\mathbb{Z}_N, \mathcal{U})\] such that $\rho_{t}^{\mathcal{U}}([q_{ij}])=\omega^{t(j-i)}[q_{ij}]$ for each $i,j \in \mathbb{Z}_N$ and $t \in \mathbb{Z}_N$, satisfying the two conditions in Definition \textup{\ref{def:znalg}}, making $(\mathcal{U},\rho^{\mathcal{U}})$ a $\mathbb{Z}_N$\nobreakdash-$\textup{C}^*$\nobreakdash-algebra.
\end{proposition}

\begin{proposition}
There is a unique unital $*$-homomorphism \[\Delta_{\mathcal{U}} : \mathcal{U} \rightarrow \mathcal{U} \boxtimes_{\textup{R}} \mathcal{U}\] such that $\Delta_{\mathcal{U}}([q_{ij}])=\sum_{k \in \mathbb{Z}_N}j_1([q_{ik}])j_2([q_{kj}])$ for $i,j \in \mathbb{Z}_N$. Furthermore, $\Delta_{\mathcal{U}}$ is $\mathbb{Z}_N$\nobreakdash-equivariant, coassociative and bisimplifiable \textup{(}see Definition \textup{\ref{def:bcqg})}.
\end{proposition}

We observe that there is a unital $*$-homomorphism $\eta^{\mathcal{U}} : \textup{C}(V) \rightarrow \textup{C}(V) \boxtimes_{\textup{R}} \mathcal{U}$ obtained as the composite

\[\textup{C}(V) \rightarrow \textup{C}(V) \boxtimes_{\textup{R}} \textup{C}(\textup{S}^+_N(\textup{R})) \rightarrow \textup{C}(V) \boxtimes_{\textup{R}} \mathcal{U},\]where the first one is obtained from $\eta^{\textup{C}(X_N)}$ after identifying $V$ with $X_N$, and the second one is obtained from the quotient map $\textup{C}(\textup{S}^+_N(\textup{R})) \rightarrow \mathcal{U}$.

\begin{corollary}
There exists an anyonic compact quantum group $G_{\Gamma}$ such that $(\textup{C}(G_{\Gamma}),\rho^{\textup{C}(G_{\Gamma})},\Delta_{G_{\Gamma}})=(\mathcal{U},\rho^{\mathcal{U}},\Delta_{\mathcal{U}})$. Furthermore, $G_{\Gamma}$ acts faithfully on $\textup{C}(V)$ by preserving the quantum symmetry of $\Gamma$ via $\eta^{\mathcal{U}}$, denoted henceforth by $\eta^{\mathrm{C}(V)}$.
\end{corollary}

\begin{proof}
  The proof is immediate and hence is omitted.
\end{proof}

\begin{theorem}\label{theo:graph-symmetry}
Let $\Gamma=(E,V)$ be a finite, simple, undirected, circulant graph. Then the anyonic quantum symmetry group $(\mathrm{Qaut}(\Gamma),\eta^{\Gamma})$ of the graph $\Gamma$ exists.
\end{theorem}

\begin{proof}
We shall show that $(G_{\Gamma},\eta^{\mathrm{C}(V)})$ is the terminal object in the category $\mathcal{C}(\Gamma)$ and thus is isomorphic to $(\mathrm{Qaut}(\Gamma),\eta^{\Gamma})$. To that end, let $(G,\eta) \in \mathrm{Obj}(\mathcal{C}(\Gamma))$ be an object in the category $\mathcal{C}(\Gamma)$. By Theorem \ref{thm:main}, there is a unique $\psi_{G} \in \mathrm{Mor}^{\mathbb{\mathbb{Z}_{N}}}(\textup{C}(\textup{S}^+_N(\textup{R})), \textup{C}(G))$ such that $(\mathrm{id}_{\textup{C}(V)} \boxtimes_{\textup{R}} \phi)\circ \eta^{\textup{C}(X_{N})}=\eta$. Now since $G$ acts on $\Gamma$ by preserving its quantum symmetry, $\psi_{G}$ descends to a unique $\phi_{G} \in \mathrm{Mor}^{\mathbb{\mathbb{Z}_{N}}}(\textup{C}(G_{\Gamma}), \textup{C}(G))$ such that $(\mathrm{id}_{\textup{C}(V)} \boxtimes_{\textup{R}} \phi)\circ \eta^{\mathrm{C}(V)}=\eta$. This is equivalent to saying that $G_{\Gamma}$ is indeed the terminal object in the category $\mathcal{C}(\Gamma)$.
\end{proof}

Now that we have the existence of $\mathrm{Qaut}(\Gamma)$ established, we compute it in some simple examples.

\begin{example}[Complete graphs]
For the complete graph $\Gamma=K_{N}$ on $N$ vertices, we observe that $(\Omega A_{K_{N}}\Omega^{-1})_{00}=N-1$ and $(\Omega A_{K_{N}}\Omega^{-1})_{ii}=-1$ for 
 all $i \in \mathbb{Z}_{N} \setminus \{0\}$, consequently, $[q,\Omega
 A_{\Gamma}\Omega^{-1}]=0$. Therefore, as in the unbraided situation,
 $\mathrm{Qaut}(\Gamma)\cong\textup{S}_{N}^{+}(\textup{R})$. In particular,
 $\mathrm{Qaut}(C_{N})\cong\textup{S}_{N}^{+}(\textup{R})$ for $N=2,3$, where $C_{N}$ is undirected $N$-cycle.
\end{example}

\begin{remark}
Let us write  $\Gamma'$ for the complement graph of $\Gamma$. Then $A_{\Gamma'}=A_{K_{N}}-A_{\Gamma}$. Now since $q$ already commutes with $A_{K_{N}}$, it commutes with $A_{\Gamma}$ if and only if it commutes with $A_{\Gamma'}$. Thus we see that $\mathrm{Qaut}(\Gamma')$ exists and is isomorphic to $\mathrm{Qaut}(\Gamma)$. In particular $X_{N}=K_{N}'$, so that $\mathrm{Qaut}(X_{N})\cong\mathrm{Qaut}(K_{N})\cong\textup{S}_{N}^{+}(\textup{R})$.
\end{remark}

\begin{example}[Graphs formed by line segments]
Consider the graph $\Gamma$ formed by $N$ line segments. More precisely, $|V|=2N$, $|E|=N$ and the adjacency matrix is given by:
\[
  A_{\Gamma}=
  \left(
   \begin{array}{cc} 
   0_{N} & I_{N}\\
   I_{N} & 0_{N}
   \end{array}
   \right). 
\]
Indeed, $A_{\Gamma}$ is the symmetric circulant matrix with $a_{k}=0$ for
$k\in\mathbb{Z}_{2N}\setminus\{N\}$ and $a_{N}=1$, so that $(\Omega
A_{\Gamma}\Omega^{-1})_{ii}=\omega^{iN}=(-1)^{i}$ for
$i\in\mathbb{Z}_{2N}$. Then $q\Omega A_{\Gamma}\Omega^{-1}=\Omega
A_{\Gamma}\Omega^{-1}q$ yields $((-1)^{i}-(-1)^{j})q_{ij}=0$, which forces
$q_{ij}=0$ whenever $i\pm j$ is odd, for all $i,j\in\mathbb{Z}_{2N}$. In the
unbraided setting, the quantum symmetry group of $\Gamma$ is the hyperoctahedral
quantum group $H_{N}^{+}$ constructed in \cite{BBC2007}. However, we get a
different presentation of $H_{N}^{+}$. The $\textup{C}^{*}$-algebra 
$\textup{C}(H_{N}^{+})$ is generated by $q_{ij}$, for $i,j \in \mathbb{Z}_{2N}$
satisfying \eqref{eq:rel-ord} and $q_{ij}=0$, whenever $i\pm j$ is odd. The
comultiplication is given by $$ q_{ij} \mapsto \sum_{k \in \mathbb{Z}_{2N}}
q_{ik} \otimes q_{kj}, \quad i,j \in \mathbb{Z}_{2N}. $$Thus in our setting, it is natural to and we
do call the resulting $\mathrm{Qaut}(\Gamma)$ as the anyonic hyperoctahedral
quantum group and we denote it by $H_{N}^{+}(\textup{R})$. Following the
arguments used in the proof of Proposition \ref{prop:lowercases}, we can show that
$\textup{C}(H_{N}^{+}(\textup{R}))$ is a noncommutative and infinite dimensional
$\textup{C}^{*}$-algebra, for all $N\geq 2$. We note, however, that the
receptacle of the comultiplication is the commutative tensor product
$\textup{C}(H_{2}^{+}(\textup{R}))\otimes\textup{C}(H_{2}^{+}(\textup{R}))$ for
$N=2$ and the braided tensor product
$\textup{C}(H_{N}^{+}(\textup{R}))\boxtimes_{\textup{R}}\textup{C}(H_{N}^{+}(\textup{R}))$, for all $N\geq 3$.
Indeed, for the former, we note that
$$
j_{1}(q_{ij})j_{2}(q_{kl})=\omega^{(j-i)(l-k)}j_{2}(q_{kl})j_{1}(q_{ij})=j_{2}(q_
{kl})j_{1}(q_{ij}), 
$$since $i\pm j$ and $k\pm l$ are even in $\mathbb{Z}_{4}$, which says that
$H_{2}^{+}(\textup{R})$ is an ordinary compact quantum group and for the latter,
the arguments are similar to the one in Remark \ref{rem:outer-NC}. We also
observe that for $N=2$, $\Gamma'\cong C_{4}$, where $C_{4}$ is the
undirected cyclic graph on $4$ vertices (i.e., a square), hence
$\mathrm{Qaut}(C_{4})\cong H_{2}^{+}(\textup{R})$.  
\end{example}

One may ask what is the anyonic quantum symmetry group of an undirected
$N$-cycle $C_{N}$, when $N \geq 5$? The following result answers this question and generalizes \cite{banicagraphs}*{Theorem 4.2}.

\begin{theorem}\label{theo:dihedral}
Let $\Gamma$ be a circulant graph with $N \geq 5$
vertices. Consider the polynomial $Q(z)=\sum_{i\in\mathbb{Z}_{N}}a_{i}z^{i}$,
where $a_{i}$ is the element in the $i$-th column of the first row of the
adjacency matrix $A_{\Gamma}$. Suppose the numbers 
$$ 
Q(1),Q(\omega),\dots,Q(\omega^{[\frac{N}{2}]}) 
$$are distinct. Then $\textup{C}(\mathrm{Qaut}(\Gamma))\cong\textup{C}(D_{N})$,
where $D_{N}$ is the dihedral group, and the comultiplication map
$\Delta_{\mathrm{Qaut}(\Gamma)}$ takes values in $\textup{C}(\mathrm{Qaut}(\Gamma))\boxtimes_{\textup{R}}\textup{C}(\mathrm{Qaut}(\Gamma))$.
\end{theorem}

\begin{proof}
The approach is more or less similar to that of \cite{banicagraphs}*{Theorem
  4.2} and so we will be brief. Firstly, we observe that
$Q(\omega^{i})=Q(\omega^{-i})$ for all $i \in\mathbb{Z}_{N}$. Now since
$$ 
\Omega A_{\Gamma}\Omega^{-1}=\textup{diag}(Q(1), Q(\omega),\cdots , Q(\omega^{N-1})),
$$and $Q(1),Q(\omega),\dots,Q(\omega^{[\frac{N}{2}]})$ are distinct,
$q\Omega A_{\Gamma}\Omega^{-1}=\Omega A_{\Gamma}\Omega^{-1}q$ yields
$(Q(\omega^{i})-Q(\omega^{j}))q_{ij}=0,$ for all $i,j \in \mathbb{Z}_{N}$, forcing $q_{ij}=0$, whenever $i \pm j\neq 0$ in $\mathbb{Z}_{N}.$ 
 
By the definition of $\eta^{\Gamma}$, for any $i=0,\dots,N-2\in\mathbb{Z}_{N}$, we have 
$$  
\eta^{\Gamma}(P_{i+1})=\sum_{l\in\mathbb{Z}_{N}}j_{1}(P_{l})j_{2}(q_{l,i+1})=\sum_{l=\pm (i+1)}j_{1}(P_{l})j_{2}(q_{l,i+1}).
$$
Using Eq.\eqref{eq:P's}, we get $\frac{1}{N}\eta^{\Gamma}(P_{i+1})
=\eta^{\Gamma}(P_{i})\eta^{\Gamma}(P_{1})$, which in turn implies 
$$  
  q_{i+1,i+1}=q_{ii}q_{11}, 
  \quad 
  q_{-i-1,i+1}=\omega^{2i}q_{-i,i}q_{-1,1},
  \quad 
  q_{-i,i}q_{11}=0,
  \quad
  q_{ii}q_{-1,1}=0,
$$for any $i=0,\dots,N-2\in\mathbb{Z}_{N}$. The first two equations recursively
yield
$$  
q_{(i+1),(i+1)}=q_{11}^{i+1}, \quad q_{-(i+1),(i+1)}=\omega^{i(i+1)}q_{-1,1}^{(i+1)}
$$for $i=0,\cdots,N-2\in\mathbb{Z}_{N}$ and the remaining equations yield
$$  
q_{11}q_{-1,1}=0=q_{-1,1}q_{11}.
$$Using the third relation in Definition~\ref{def:permutationalg}, we obtain
$$  
q_{11}^{*}=q_{-1,-1}=q_{11}^{N-1}
$$and
$$  
q_{-1,1}^{*}=\omega^{2}q_{1,-1}=\omega^{2}q_{-(N-2+1),(N-2+1)}=q_{-1,1}^{N-1},
$$and therefore
$$  
q_{11}^{*}q_{-1,1}=0=q_{-1,1}q_{11}^{*}.
$$This shows that $\textup{C}(\mathrm{Qaut}(\Gamma))$ is isomorphic to
$\textup{C}(D_{N})$. However, the comultiplication
$\Delta_{\mathrm{Qaut}(\Gamma)}$ still takes values in the braided tensor
product
$\textup{C}(\mathrm{Qaut}(\Gamma))\boxtimes_{\textup{R}}\textup{C}(\mathrm{Qaut}(\Gamma))$. To
see this, we choose $i\in\mathbb{Z}_{N}$ with $\omega^{-4i^{2}}\neq 1$, so that 
$$
\Delta_{\Gamma}(q_{ii})=j_{1}(q_{i,-i})j_{2}(q_{-i,i})+j_{1}(q_{ii})j_{2}(q_{ii}). 
$$Since $q_{i,-i}$ and $q_{-i,i}$ are homogeneous of degrees $-2i$ and $2i$,
respectively, we obtain
$j_{1}(q_{i,-i})j_{2}(q_{-i,i})=\omega^{-4i^{2}}j_{2}(q_{-i,i})j_{1}(q_{i,-i}).$
\end{proof}

\begin{corollary}\label{cor:cyclic-graph}
For the undirected $N$-cycle $C_{N}$ with $N\geq 2$, we have the following.
\begin{enumerate}
\item For $N=2,3$, $\mathrm{Qaut}(C_{N})\cong\textup{S}^{+}_{N}(\textup{R})$.
\item For $N=4$, $\mathrm{Qaut}(C_{4})\cong H^{+}_{2}(\textup{R})$.
\item For $N\geq 5$, $\textup{C}(\mathrm{Qaut}(C_{N}))\cong\textup{C}(D_{N})$ and
$$  
\Delta_{\mathrm{Qaut}(C_{N})} \in \textup{Mor}^{\mathbb{Z}_{N}}(\textup{C}(\mathrm{Qaut}(C_{N})),\textup{C}(\mathrm{Qaut}(C_{N}))\boxtimes_{\textup{R}}\textup{C}(\mathrm{Qaut}(C_{N}))).
$$
\end{enumerate}
\end{corollary}

\appendix
\section{}\label{appendix:boso} In this appendix, we discuss the bosonization construction which gives an equivalence between the category of anyonic compact quantum groups and the category of ordinary compact quantum groups together with an idempotent quantum group homomorphism. Using this, we describe the representation theory of the anyonic quantum permutation group $\textup{S}^+_{N}(\mathrm{R})$. As this is mostly similar to the results in \cites{MR2021,BJR2022} and is needed at two places en route to our main theorem, our discussion will be brief, only highlighting the crucial points. We begin with recalling a few necessary preliminaries.

\begin{proposition}\label{prop:psi}\cite{mrw2016}
Let $(X,\rho^X)$ and $(Y,\rho^Y)$ be two $\mathbb{Z}_N$\nobreakdash-$\textup{C}^*$\nobreakdash-algebras. Then there is a unique morphism \[\psi^{X,Y} \in \mathrm{Mor}(\textup{C}(\mathbb{Z}_N) \boxtimes_{\textup{R}} X \boxtimes_{\textup{R}} Y, (\textup{C}(\mathbb{Z}_N) \boxtimes_{\textup{R}} X) \otimes (\textup{C}(\mathbb{Z}_N) \boxtimes_{\textup{R}} Y))\]
such that
\allowdisplaybreaks{
\begin{align*}
\psi^{X,Y}(j_1(x))={}&(j_1 \otimes j_1)\Delta_{\mathbb{Z}_N}(x),\\
\psi^{X,Y}(j_2(a))={}&(j_2 \otimes j_1)(\rho^{X}(a)),\\
\psi^{X,Y}(j_3(b))={}&1_{\textup{C}(\mathbb{Z}_N) \boxtimes_{\textup{R}} X} \otimes j_2(b),
\end{align*}
}for $x \in \textup{C}(\mathbb{Z}_N)$, $a \in X$, and $b \in Y$.    
\end{proposition}

With $\psi$ in hand, we can now recall the definition of the bosonization of $G$.

\begin{proposition}\label{prop:bos}\cite{mrw2016}
Let $G=(\textup{C}(G),\rho^{\textup{C}(G)},\Delta_G)$ be an anyonic compact quantum group. Then the pair $(\textup{C}(\mathbb{Z}_N)\boxtimes_{\textup{R}} \textup{C}(G), \psi^{\textup{C}(G),\textup{C}(G)}\circ (\mathrm{id}_{\textup{C}(\mathbb{Z}_N)} \boxtimes_{\textup{R}} \Delta_G))$ satisfies the axioms for a compact quantum group, called the bosonization of $G$ and denoted by $G \rtimes \mathbb{Z}_N=(\textup{C}(G \rtimes \mathbb{Z}_N),\Delta_{G \rtimes \mathbb{Z}_N})$.
\end{proposition}

The next theorem describes the bosonization of $\textup{S}_N^+(\textup{R})$ explicitly; the proof is similar to the proof of \cite{BJR2022}*{Theorem 3.3}

\begin{theorem}\label{thm:bosopermutation}
Let $\textup{S}_{N}^+(\textup{R}) \rtimes \mathbb{Z}_N=(\textup{C}(\textup{S}_{N}^+(\textup{R}) \rtimes \mathbb{Z}_N),\Delta_{\textup{S}_{N}^+(\textup{R}) \rtimes \mathbb{Z}_N})$ be the bosonization of the anyonic quantum permutation group $\textup{S}_{N}^+(\textup{R})$. Then $\textup{C}(\textup{S}_{N}^+(\textup{R}) \rtimes \mathbb{Z}_N)$ is the universal unital $\textup{C}^*$\nobreakdash-algebra generated by elements $z$ and $q_{ij}$ for $i,j \in \mathbb{Z}_N$ subject to
\begin{enumerate}
    \item the relations $zz^*=z^*z=1=z^N$,
    \item the commutation relations $zq_{ij}=\omega^{j-i}q_{ij}z$, for $i,j \in \mathbb{Z}_N$,
    \item and the relations in Definition \textup{\ref{def:permutationalg}}.
\end{enumerate}
Furthermore, the comultiplication $\Delta_{\textup{S}_{N}^+(\textup{R}) \rtimes \mathbb{Z}_N}$ is given by
\begin{equation}
\Delta_{\textup{S}_{N}^+(\textup{R}) \rtimes \mathbb{Z}_N}(z)=z \otimes z, \quad \Delta_{\textup{S}_{N}^+(\textup{R}) \rtimes \mathbb{Z}_N}(q_{ij})=\sum_{k \in \mathbb{Z}_N} q_{ik} \otimes z^{k-i}q_{kj},
\end{equation} for $i,j \in \mathbb{Z}_N$.
\end{theorem}

We next describe the bosonization of the anyonic free unitary quantum group, for the sake of completeness.

\begin{theorem}\label{thm:bosounitary}
Let $\textup{U}_{N}^+(\textup{R}) \rtimes \mathbb{Z}_N=(\textup{C}(\textup{U}_{N}^+(\textup{R}) \rtimes \mathbb{Z}_N),\Delta_{\textup{U}_{N}^+(\textup{R}) \rtimes \mathbb{Z}_N})$ be the bosonization of the anyonic free unitary quantum group $\textup{U}_{N}^+(\textup{R})$. Then $\textup{C}(\textup{U}_{N}^+(\textup{R}) \rtimes \mathbb{Z}_N)$ is the universal unital $\textup{C}^*$\nobreakdash-algebra generated by elements $z$ and $u_{ij}$ for $i,j \in \mathbb{Z}_N$ subject to
\begin{enumerate}
\item the relations $zz^*=z^*z=1=z^N$,
\item the commutation relations $zu_{ij}=\omega^{j-i}u_{ij}z$, for $i,j \in \mathbb{Z}_N$,
\item and the relations in Definition \textup{\ref{def:unitaryalg}}.
\end{enumerate}
Furthermore, the comultiplication $\Delta_{\textup{U}_{N}^+(\textup{R}) \rtimes \mathbb{Z}_N}$ is given by
\begin{equation}
\Delta_{\textup{U}_{N}^+(\textup{R}) \rtimes \mathbb{Z}_N}(z)=z \otimes z, \quad \Delta_{\textup{U}_{N}^+(\textup{R}) \rtimes \mathbb{Z}_N}(u_{ij})=\sum_{k \in \mathbb{Z}_N} u_{ik} \otimes z^{k-i}u_{kj},
\end{equation} for $i,j \in \mathbb{Z}_N$.
\end{theorem}

Our next aim is to prove that the bosonization $\textup{S}_{N}^+(\textup{R}) \rtimes \mathbb{Z}_N$ is a compact matrix quantum group. For that, we need to dig slightly deeper into the representation theory of $\textup{S}_{N}^+(\textup{R})$; by the techniques in \cite{mrw2016}, it can be shown that representations of $\textup{S}_{N}^+(\textup{R})$ are equivalent to representations of the bosonization $\textup{S}_{N}^+(\textup{R}) \rtimes \mathbb{Z}_N$. A detailed study of this correspondence will be undertaken shortly. For the moment, we concentrate on the special case where the carrier Hilbert space is finite dimensional, in which case it is easier to describe this correspondence using the matrix coefficients.

\begin{proposition}\label{prop:fundamental}
Let $t_{ij}=j_1(z^{i})j_2(q_{ij}) \in \textup{C}(\textup{S}_{N}^+(\textup{R}) \rtimes \mathbb{Z}_N)=\textup{C}(\mathbb{Z}_N) \boxtimes_{\textup{R}} \textup{C}(\textup{S}_{N}^+(\textup{R}))$ for $i,j \in \mathbb{Z}_N$. Then $t=(t_{ij})_{i,j \in \mathbb{Z}_N} \in M_N(\textup{C}(\textup{S}_{N}^+(\textup{R}) \rtimes \mathbb{Z}_N))$ defines a finite dimensional unitary representation of the compact quantum group $\textup{S}_{N}^+(\textup{R}) \rtimes \mathbb{Z}_N$.   
\end{proposition}

Again, we omit the proof and refer the interested reader to \cite{BJR2022}*{Proposition 3.4}.

\begin{corollary}\label{cor:cqmgpermutation}
The pair $(\textup{C}(\textup{S}_{N}^+(\textup{R}) \rtimes \mathbb{Z}_N), z \oplus t)$ is a compact matrix quantum group.
\end{corollary}

The proof is similar to the proof of \cite{BJR2022}*{Corollary 3.5}. Again, for the sake of completeness, we state the corresponding results for the anyonic free unitary quantum groups.

\begin{proposition}\label{prop:fundamentalunitary}
Let $t'_{ij}=j_1(z^{i})j_2(u_{ij}) \in \textup{C}(\textup{U}_{N}^+(\textup{R}) \rtimes \mathbb{Z}_N)=\textup{C}(\mathbb{Z}_N) \boxtimes_{\textup{R}} \textup{C}(\textup{U}_{N}^+(\textup{R}))$ for $i,j \in \mathbb{Z}_N$. Then $t'=(t'_{ij})_{i,j \in \mathbb{Z}_N} \in M_N(\textup{C}(\textup{U}_{N}^+(\textup{R}) \rtimes \mathbb{Z}_N))$ defines a finite dimensional unitary representation of the compact quantum group $\textup{U}_{N}^+(\textup{R}) \rtimes \mathbb{Z}_N$.   
\end{proposition}
    
\begin{corollary}\label{cor:cqmgunitary}
The pair $(\textup{C}(\textup{U}_{N}^+(\textup{R}) \rtimes \mathbb{Z}_N), z \oplus t')$ is a compact matrix quantum group.
\end{corollary}

We now come to the representation theory of $\textup{C}(\textup{S}^+_{N}(\mathrm{R}) \rtimes \mathbb{Z}_{N})$. We begin by observing that the $\textup{C}^*$\nobreakdash-algebra $\textup{C}(\textup{S}^+_{N}(\mathrm{R}) \rtimes \mathbb{Z}_{N})$ is the universal unital $\textup{C}^*$\nobreakdash-algebra generated by $z$ and $t_{ij}$ for $i,j \in \mathbb{Z}_{N}$, and that the matrix $t=(t_{ij})_{i,j \in \mathbb{Z}_{N}}$ satisfies Eq.\eqref{eq:rel-ord}. Furthermore, the relation $zt_{ij}=\omega^{i-j}t_{ij}z$ and its adjoint, yield equivalences $z \tenscorep t \cong t \tenscorep z$ and $z \tenscorep \overline{t} \cong \overline{t} \tenscorep z$, respectively. 

Now we consider the function algebra $\mathrm{C}(\mathrm{S}^{+}_{N})$ of the quantum permutation group, the generators being denoted by $x_{ij}$, $i,j \in \mathbb{Z}_{N}$. Multiplying the generators $x_{ij}$ by $\omega^{i-j}$, $i,j \in \mathbb{Z}_{N}$ yields an automorphism $\alpha$ of the $\textup{C}^*$\nobreakdash-algebra $\mathrm{C}(\mathrm{S}^{+}_{N})$. Following \cite{MR2021}, there is a Hopf $*$\nobreakdash-homomorphism $\phi : \mathrm{C}(\mathrm{S}^{+}_{N}) \rightarrow \textup{C}(\textup{S}^+_{N}(\mathrm{R}) \rtimes \mathbb{Z}_{N})$ mapping $x_{ij}$ to $\sum_{r,s \in \mathbb{Z}_{N}} \frac{1}{N}\omega^{-ir+sj}t_{rs}$, i.e., sends the matrix $x=(x_{ij})_{i,j \in \mathbb{Z}_{N}}$ to $\Omega t \Omega^{-1}$. The proof of the following is similar to Proposition 4.2 of \cite{MR2021}.

\begin{proposition}
The Hopf $*$\nobreakdash-homomorphism $\phi : \mathrm{C}(\mathrm{S}^{+}_{N}) \rightarrow \textup{C}(\textup{S}^+_{N}(\mathrm{R}) \rtimes \mathbb{Z}_{N})$ extends to an isomorphism $\mathrm{C}(\mathrm{S}^{+}_{N}) \rtimes_{\alpha} \widehat{\mathbb{Z}_{N}} \cong \textup{C}(\textup{S}^+_{N}(\mathrm{R}) \rtimes \mathbb{Z}_{N})$.
\end{proposition}

The Hopf $*$\nobreakdash-homomorphism $\phi$ induces a fully faithful strict tensor functor $\phi_{\ast}$ from the representation category of $\mathrm{C}(\mathrm{S}^{+}_{N})$ to that of $\textup{S}^+_{N}(\mathrm{R}) \rtimes \mathbb{Z}_{N}$. The irreducible representations of the quantum permutation group, as described by Banica \cite{banica:generic-coaction}, are enumerated by $r_{a}$, $a \in \mathbb{N}$, with $r_{0}=1$, $r_{1}=x$ and the fusion rule is,
\begin{equation*}
r_{a} \tenscorep r_{b} \cong r_{|a-b|} \oplus r_{|a-b|+1} \oplus \dots \oplus r_{a+b}. 
\end{equation*}
A lemma analogous to Lemma 4.3 of \cite{MR2021} says that the representations $z^a \tenscorep \phi_{\ast}(r_{a'})$ with $0 \leq a \leq N-1$ and $a' \in \mathbb{N}$ are all irreducible and distinct; furthermore, any irreducible representation is of this form. And finally, the fusion rules then become
\begin{equation*}
(z^a \tenscorep r_{a'}) \tenscorep (z^b \tenscorep r_{b'})\cong z^{a+b} \tenscorep r_{|a'-b'|} \oplus r_{|a'-b'|+1} \oplus \dots \oplus r_{a'+b'}. 
\end{equation*}
We note, by following the same arguments as in \cite{MR2021}, that the representation category of $\textup{S}_{N}^+(\mathrm{R})$ is equivalent to that of the bosonization $\textup{S}_{N}^+(\mathrm{R}) \rtimes \mathbb{Z}_{N}$. Collecting these together, we end this article with the following theorem.

\begin{theorem}
The anyonic quantum permutation group $\textup{S}_{N}^+(\mathrm{R})$ has irreducible representations $r_{(a,a')}$, $0 \leq a \leq N-1$, $a' \in \mathbb{N}$ such that any irreducible representation is unitarily equivalent to exactly one of these and moreover they satisfy the fusion rule
\begin{equation*}
r_{(a,a')} \tenscorep r_{(b,b')} \cong r_{(a+b,|a'-b'|)} \oplus r_{(a+b,|a'-b'|+1)} \oplus \dots \oplus r_{(a+b,a'+b')}. 
\end{equation*}
\end{theorem}

\end{document}